\documentclass[10pt,twocolumn,letterpaper]{article}

\usepackage{cvpr}
\usepackage{times}
\usepackage{epsfig}
\usepackage{epstopdf}
\usepackage{graphicx}
\usepackage{enumerate}
\usepackage{amsmath}
\usepackage{amssymb}
\usepackage[ruled,vlined]{algorithm2e}
\usepackage{bbm}
\usepackage{listings}
\usepackage{subcaption}
\usepackage{kantlipsum} 
\usepackage{mwe} 

\usepackage{amsthm}

\newtheorem{theorem}{Theorem}
\newtheorem{lemma}{Lemma}

\newcommand{\R}{\mathbb{R}}

\usepackage{hyperref}

\cvprfinalcopy 

\def\cvprPaperID{4321} 

\ifcvprfinal
\begin{document}

\title{Online and Batch Supervised Background Estimation via L1 Regression}

\author{Aritra Dutta\\
	KAUST\\
	{\tt\small aritra.dutta@kaust.edu.sa}
	\and
	Peter Richt\'{a}rik\\
	KAUST, Edinburgh, MIPT\\
	{\tt\small peter.richtarik@kaust.edu.sa}
}

\maketitle

\begin{abstract} We propose a surprisingly simple model for supervised video background estimation. Our model is based on $\ell_1$ regression. As existing methods for $\ell_1$ regression do not scale to high-resolution videos, we propose several simple and scalable methods for solving the problem, including iteratively reweighted least squares, a homotopy method, and stochastic gradient descent. We show through extensive experiments that our model and methods match or outperform the state-of-the-art online and batch methods in virtually all quantitative and qualitative measures.
\end{abstract}

\section{Introduction}
Video background estimation and moving object detection is a classic problem in computer vision. Among several existing approaches, one of the most prevalent ones is to solve it in a matrix decomposition framework~\cite{Bouwmans201431,Bouwmans2016}. Let $A\in \R^{m\times n'}$ be  a matrix encoding $n'$ video frames, each represented as a vector of size $m$.  Our task is to decompose all frames of the video  into background and foreground frames: $A  = B+ F.$

As described above, the problem is ill-posed, and more information about the structure of the decomposition is needed. In practice, background videos are often static or close to static, which typically means that $B$ is of low rank~\cite{oliver1999}. On the other hand, foreground usually represents objects occasionally moving across the foreground, which typically means that $F$ is sparse. These and similar observations leads to the development of models of the form~\cite{Bouwmans2016, Bouwmans201431, APG,LinChenMa,duttaligongshah}:
\begin{equation}\label{eq:general}\min_{B} f_{\rm rank}(B) + f_{\rm spar}(A-B),\end{equation}
where $f_{\rm rank}$ is a suitable function that encourages the rank of $B$ to be low, and $f_{\rm spar}$ is a suitable function that encourages the foreground $F$ to be sparse. 

Xin \etal~\cite{xin2015} recently proposed a background estimation model---generalized fused lasso~(GFL)---arising as a special case of \cite{golub} with the choice $f_{\rm rank}(B) = {\rm rank}(B)$ and $f_{\rm spar}(F) = \lambda \|F\|_{\rm GFL}$:  
\begin{align}\label{gfl}
\min_{B}{\rm rank}(B)+\lambda\|A-B\|_{\rm GFL}.
\end{align}
In this model, $\|\cdot\|_{GFL}$ is the ``generalized fused lasso'' norm, which arises from the combination of the $\ell_1$ norm (to encourage sparsity) and a local spatial total variation norm (to encourage connectivity of the foreground).

\paragraph{Supervised background estimation.} In the modern world, {\em supervised} background estimation models play an important role in the analysis of the data captured from the surveillance cameras. As the name suggests, these models rely on  prior availability of some ``training'' background frames, $B_1 \in \R^{m\times r}$. Without loss of generality, assume that  the training background frames correspond to the first $r$ frames of $B$, i.e., $B = [B_1\;B_2]$, where $B_1\in \R^{m\times r}$ is known and $B_2\in \R^{m\times n}$ is to be determined, with $n'=r+n$. Let $A=[A_1\;A_2]$ be partitioned accordingly, and let $F_2 = A_2-B_2 \in \R^{m\times n}$. In this setting, \cite{xin2015} further specialized the model \eqref{gfl} by adding the extra assumption that ${\rm rank}(B)={\rm rank}(B_1)$.
As a result, the columns of the unknown matrix $B_2$ can be written as a linear combinations of the columns of $B_1$. Specifically, $B_2$ can be written as $B_1S$, where $S\in \R^{r\times n}$ is a coefficient matrix. Thus, problem (\ref{gfl}) can be written in the form
\begin{align}\label{gfl_sfl}
&\min_{S'}{\rm rank}(B_1[I\;S'])+\lambda\|A_2 - B_1S'\|_{\rm GFL}.
\end{align}
While \eqref{gfl_golub} is the the problem Xin \etal~\cite{xin2015}  {\em wanted} to solve, they did not tackle it directly and instead further assumed that $S$ is sparse, and solved the modified problem
\begin{align}\label{gfl_sfl_sparse}
&\min_{S'}\|S'\|_{1}+\lambda\|A_2 - B_1 S' \|_{\rm GFL},
\end{align}
where $\|\cdot\|_{1}$ denotes the $\ell_1$ norm of matrices.

\section{New Model}

In this paper we propose a new supervised background estimation model, one that we argue is much better than \eqref{gfl_sfl_sparse} in several aspects. Moreover, our model and the methods we propose significantly outperform other state-of-the-art methods. 

\paragraph{L1 regression.} As in \eqref{gfl_sfl_sparse}, our model is also based on a modified version of \eqref{gfl_sfl}. We do not need to assume any sparsity on $S'$, and instead make the trivial observation that ${\rm rank}(B_1[I \;S']) = {\rm rank}(B_1)$. Since $B_1$ is known, the first term in the objective function \eqref{gfl_sfl} is constant, and hence does not contribute to the optimization problem. Hence we may drop it. Moreover, we suggest replacing the GFL norm by the $\ell_1$ norm. This leads to a very simple {\em L1 (robust) regression problem}:
\begin{equation} \min_{S' \in \R^{r\times n}} \|A_2-B_1 S'\|_1. \label{model:L1-version1}
\end{equation}

\paragraph{Dimension reduction.} The above model can be further simplified. It may be the case that the rank of $B_1 \in \R^{m\times r}$ is smaller\footnote{If this is not the case, it still may be the case that the column space of $B_1$ can be very well approximated by a space with less or much less than $r$ dimensions.} (or much smaller) than $r$.  In such a situation, we can replace  $B_1$  in \eqref{model:L1-version1} by a thinner matrix, which allows us to reduce the dimension of the optimization variable $S'$. In particular, let $B_1=QR$ 
 be the QR decomposition of $B_1$, where $Q\in \R^{m\times k}$, $R\in \R^{k\times r}$, $k={\rm rank}(B_1)$, and $Q$ has orthonormal columns.  Since the column space of $B_1$ is the same as the column space of $Q$, by using the substitution $B_1 S' = Q S $, we can reformulate \eqref{model:L1-version1} as the {\em lower-dimensional }  L1 regression problem:
 \begin{equation} \boxed{ \min_{S \in \R^{k\times n}} f(S) :=\|A_2-Q S\|_1} \label{dutta_richtarik_l1}
\end{equation}
 \begin{figure}
    \centering
    \includegraphics[width=.5\textwidth]{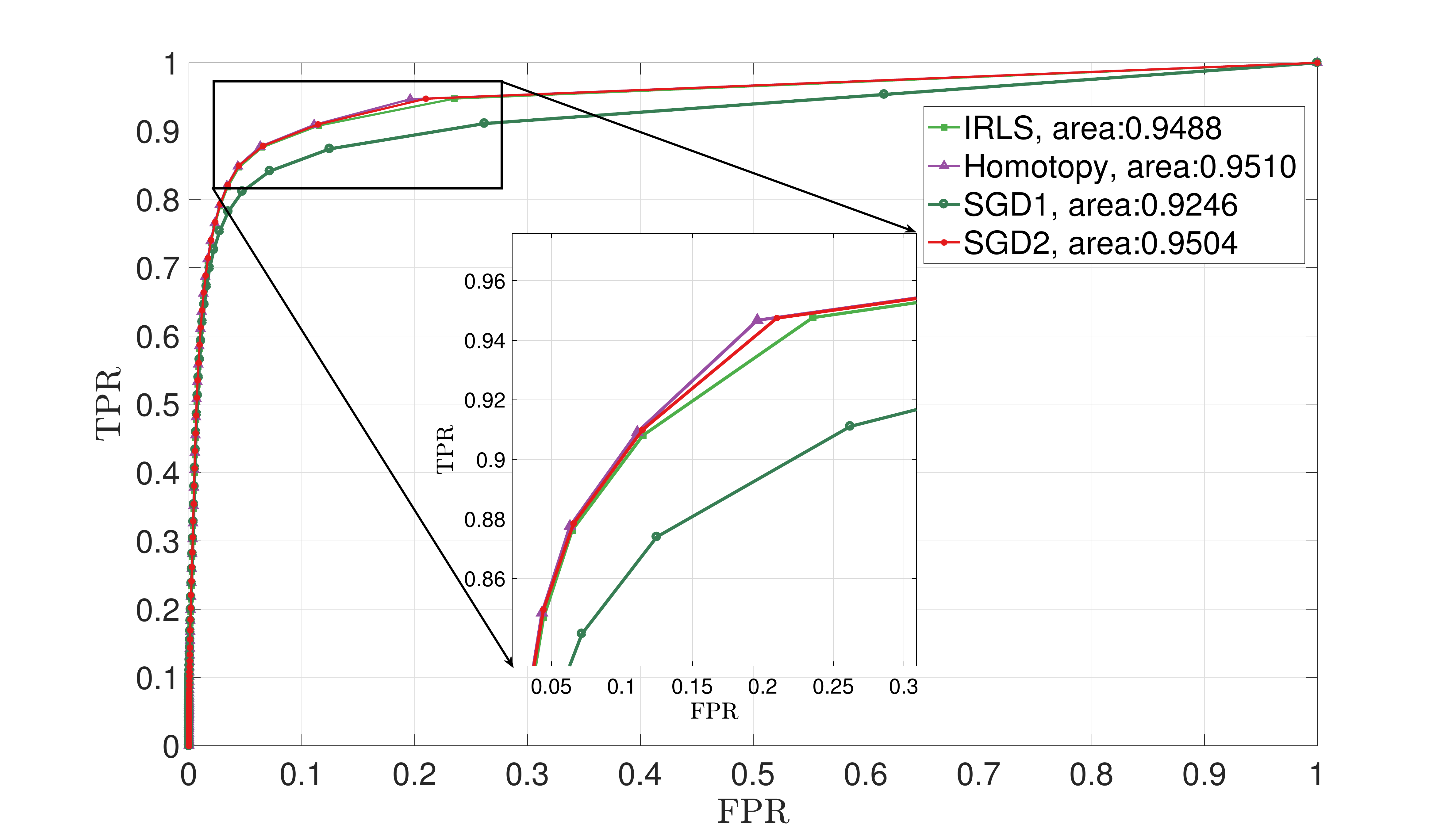}
    \caption{\small{ROC curve to compare between our proposed $\ell_1$ regression algorithms on {\tt Basic} video, frame size $144\times176$.}}
    \label{compare_methods}
\end{figure}
\paragraph{Decomposition.}  Let $A_2 = [a_1,\dots,a_n]$ and $S = [s_1,\dots,s_n]$, where $a_i \in \R^m$, $s_i \in \R^t$ for all $i\in [n]:=\{1,2,\dots,n\}$. Our model \eqref{dutta_richtarik_l1} can be decomposed into $n$ parts, one for each frame:
  \begin{equation}\label{dutta_richtarik_l1: decompose}
f(S)= \sum_{i=1}^n f_i(s_i), \quad  f_i(s_i) := \|a_i - Q s_i\|_{1},
 \end{equation}
 where $\|\cdot\|_1$ is the vector $\ell_1$ norm. Therefore,   \eqref{dutta_richtarik_l1} reduces to  $n$ small ($k
$-dimensional) and independent  {\em $\ell_1$ regression problems:}
\begin{equation}\label{dutta_richtarik_l1: decompose_problems}
\boxed{\min_{s_i\in \R^t} f_i(s_i), \qquad  i \in [n]}
 \end{equation}

\paragraph{Advantages of our model.}We now list some advantages of our model \eqref{dutta_richtarik_l1} as compared to \eqref{gfl_sfl_sparse}. We show that 1) our model does not involve the unnecessary sparsity inducing term $\|S'\|_1$, that  2) our model does not include the trade-off parameter $\lambda$ and hence issues with  tuning this parameter disappear,  that 3) our model involves a simple  $\ell_1$ norm as opposed to the more complicated GFL norm, that 4) the dimension of $S$ is smaller (and possibly much smaller) than that of $S'$, that 5) our objective is separable across the $n$ columns  of $S$ corresponding to frames, which means that we can solve for each column of $S$  in {\em parallel} (for instance on a GPU), and that 6) for the same reason, we can solve for each frame as it arrives, in an {\em online} fashion.

\paragraph{Further contributions.}
Our model works well with just a  few training background frames (e.g., $r=10$). This should be compared with the 200 training frames in GFL model. We propose 5 methods for solving the model, out of which 4 can work online and all 5 can work in a  batch mode. Our model solves all the following challenges: static and semi-static foreground, newly added static foreground, shadows that are already present in the background and newly created by moving foreground, occlusion and disocclusion of the static and dynamic foreground, the ghosting effect of the foreground in the background. To the best of our knowledge, no other  algorithm can solve all the above challenges in a single framework.

\section{Scalable Algorithms for L1 Regression}

The  separable (across frames) structure of our model allows us to devise both batch and online background estimation algorithms. To the best of our knowledge, this is the first formulation which can operate in both batch and online mode. Since our problem decomposes across frames $i\in [n]$, it suffices to describe algorithms for solving the $\ell_1$ regression problem \eqref{dutta_richtarik_l1: decompose_problems} for a single $i$. This problem has the form
\begin{equation}\label{dutta_richtarik_l1: simple}
 \min_{x\in \R^t} \phi(x) := \|Q x - b\|_1 = \sum_{j=1}^m |q_j^\top x - b_j|,
 \end{equation}
 where $x \in \R^t$ corresponds to one of the reconstruction  vectors $s_i$, and $b\in\R^m$ corresponds to the related frame $a_i$. We write $b = (b_1,\dots,b_m)\in \R^m$, and let $q_j \in \R^t$ be the  $j$th row of $Q$  for $j\in [m]$.
 \begin{figure}
    \centering
    \includegraphics[width=2.6in]{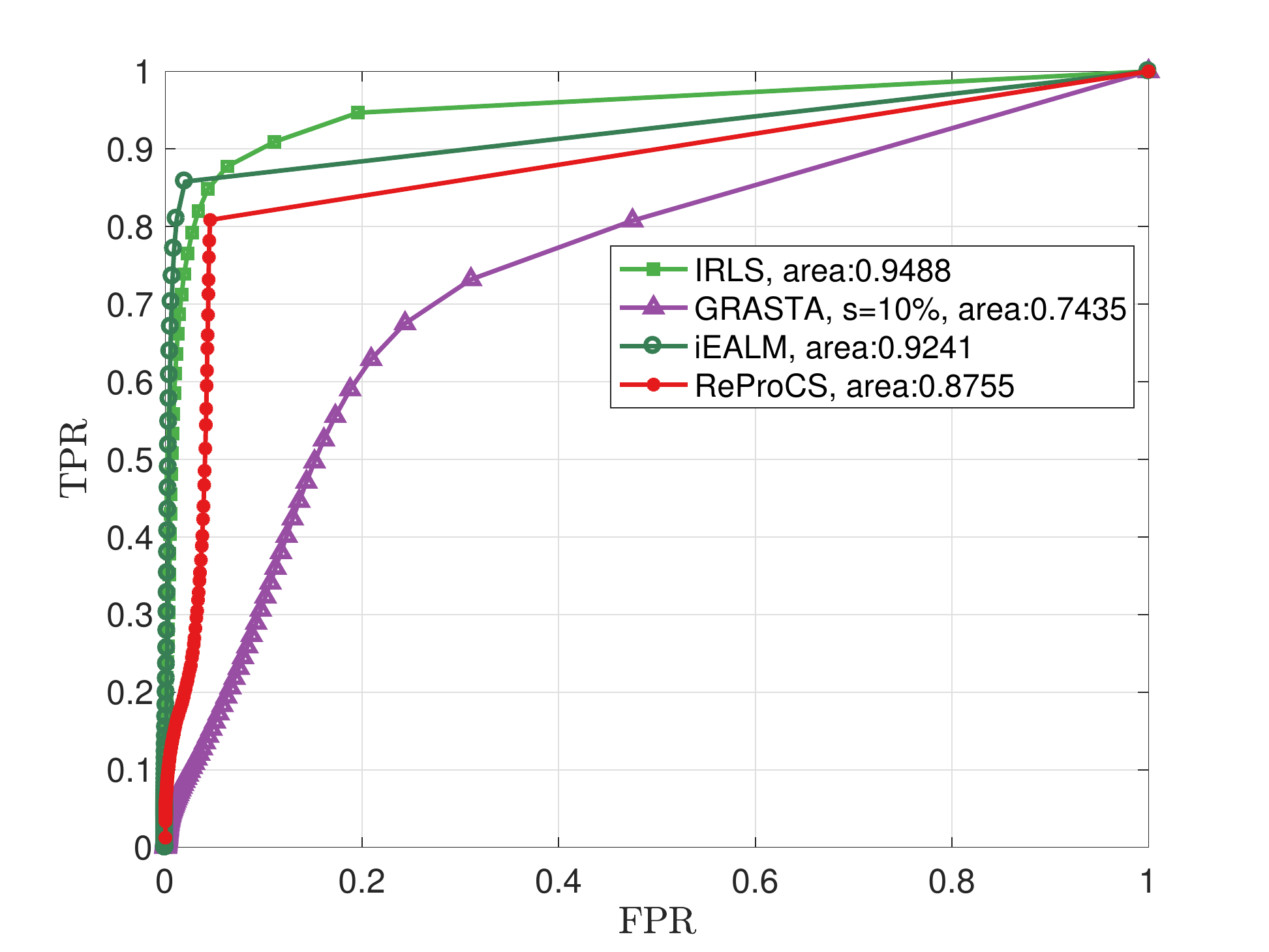}
    \caption{\small{ROC curve to compare between IRLS, iEALM, GRASTA, and ReProCS on {\tt Basic} video, frame size $144\times176$.}}
    \label{ROC_methods}
\end{figure}

\paragraph{Five methods.} In this work  we propose to solve \eqref{dutta_richtarik_l1: simple} via four  algorithms:~(a)~iteratively reweighted least squares~(IRLS), (b)~homotopy method,~(c)~stochastic subgradient descent (variant 1), (d)~stochastic subgradient descent (variant 2), and (e) Augmented Lagrangian Method of Multipliers~(ALM)~(see Appendix 2).

The first four algorithms can be used in both batch and online setting and can deal with grayscale and color images. If we assume the camera is static, and assume constant illumination throughout the video sequence, then our online methods can provide a good estimate of the background. Moreover, all algorithms are robust to the intermittent object motion artifacts, that is, static foreground (whenever a foreground object stops moving for a few frames), which poses a big challenge to the state-of-the-art methods. Additionally, our online methods are fast as we perform neither conventional nor incremental principal component analysis~(PCA). In contrast, conventional PCA~\cite{pca} is an essential subproblem to numerically solve both RPCA and GFL problems. In these problems, each iteration involves computing  PCA, which operates at a cost~$\mathcal{O}(mn^2)$ and is due to SVD on a $m\times n$ matrix. We also recall that the state-of-the-art online, semi-online, or batch incremental algorithms, such as the Grassmannian robust adaptive subspace estimation~(GRASTA)~\cite{grasta}, recursive projected compressive sensing algorithm~(ReProCS)~\cite{reprocs,pracreprocs,modified_cs}, or incremental principal component pursuit~(incPCP)~\cite{incpcp,matlab_pcp,inpcp_jitter}, use either thin or partial PCA as well.

\paragraph{The need for simpler solvers for $\ell_1$ regression.} It is natural to ask: why do we need a new set of algorithms to solve the classical $\ell_1$ regression problem when there are several well known solvers, for example, CVX~\cite{cvx,gb08}, $\ell_1$ magic~\cite{l1magic}, and SparseLab 2.1-core~\cite{sparselab}? It turns out that a high resolution video sequence (characterized by very large $m$) is computationally extremely expensive for the above mentioned classic solvers. Moreover, we do not need highly accurate solutions. Hence, simple and scalable methods are preferable to more involved and computationally demanding methods. The $\ell_1$ magic software, for example, in our experiments took  126 minutes (on a computer with Intel i7 Processor and 16 GB memory) to estimate the background on the {\tt Waving Tree} dataset with $A_2\in\mathbb{R}^{19,200\times 66}$.  In contrast, our IRLS  method took 0.59 seconds only for 66 frames.

\begin{figure}
    \centering
    \includegraphics[width=2.5in, height = 2.0 in]{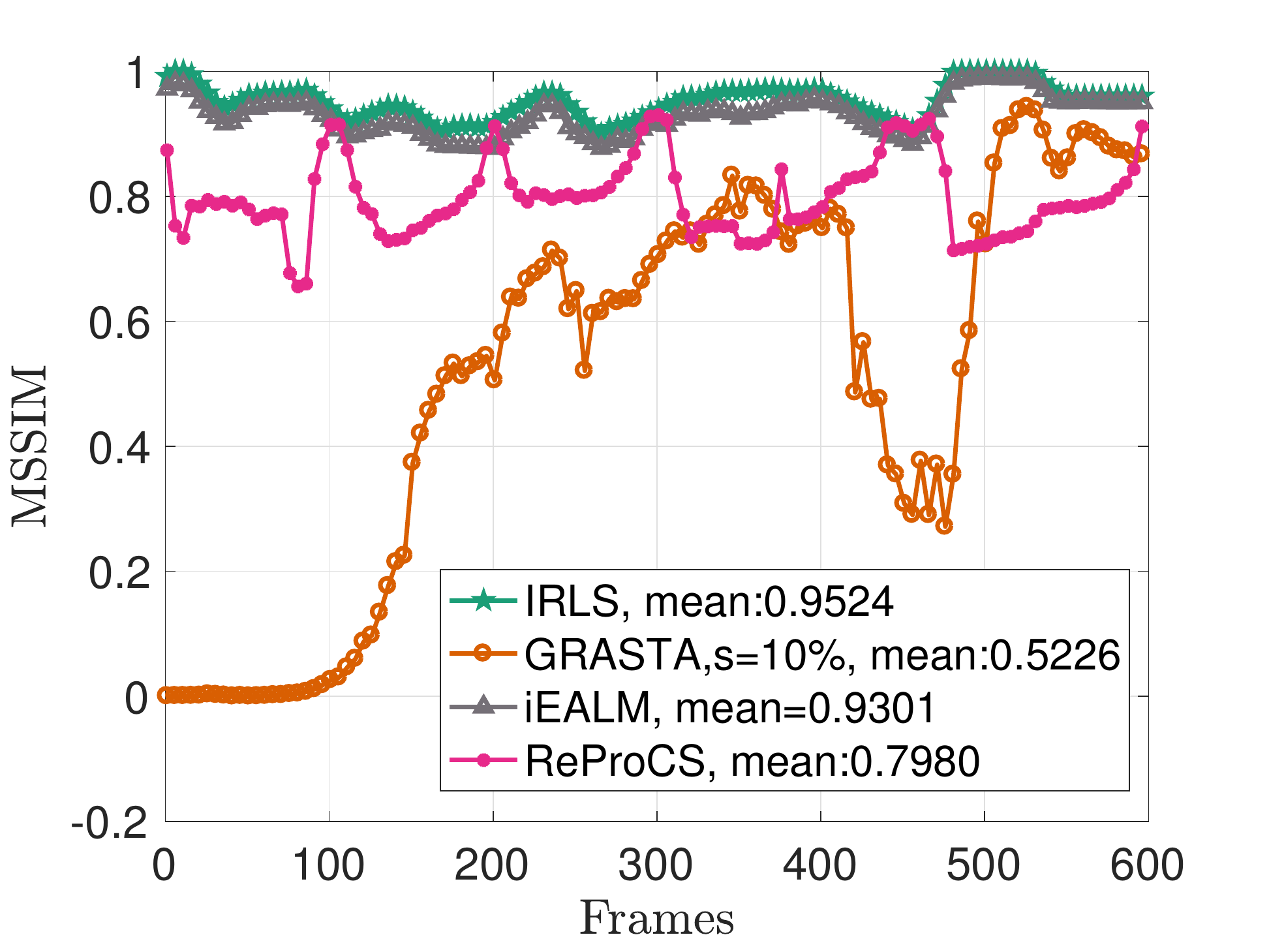}
    \caption{\small{Comparison of Mean SSIM~(MSSIM) of IRLS, iEALM, GRASTA, and ReProCS on~{\tt Basic} video.~IRLS has the best MSSIM.~To process 600 frames each of size $144\times176$, iEALM takes 164.03 seconds, GRASTA takes 20.25 seconds, ReProCS takes \textcolor{blue}{14.20} seconds, and our IRLS takes~\textcolor{red}{7.51} seconds.}}
    \label{ssim}
\end{figure}

\subsection{Iteratively Reweighted Least Squares~(IRLS)}\label{IRLS:sec}

In the past decade, IRLS  has been used in various domains, ranging from reconstruction of sparse signals from underdetermined systems, to the low-rank and sparse matrix minimization problems in face clustering, motion segmentation, filter design, automatic target detection, to mention just a few applications~\cite{richtarik_thesis, Burrus_homotopy, Candes08, ingrid, osborne, lu_irls,iirls}. We find that the IRLS algorithm is a good fit to solve~\eqref{dutta_richtarik_l1: simple}. Also,  each iteration of IRLS reduces to a single weighted $\ell_2$ regression problem for an over determined system. To the best of our knowledge, we are the first to use IRLS to propose a background estimation model. 

We now briefly describe IRLS for solving \eqref{dutta_richtarik_l1: simple}. First note that the cost function $f$ in  \eqref{dutta_richtarik_l1: simple} can be written in the form
\begin{equation}\label{dutta_richtarik_rls: decompose_problems}
\phi(x) = \sum_{j=1}^m | q_j^\top x - b_j | = \sum_{j=1}^m \frac{( q_j^\top x - b_j )^2 }{ | q_j^\top x - b_j| }.
 \end{equation}
For $x\in \R^m$ and $\delta>0$ define a diagonal weight matrix via 
$W_\delta(x):={\rm Diag}(1 / \max \{ | q_j^\top x - b_j | , \delta \} ).$ Given a current iterate $x_k$, we may fix the denominator in  \eqref{dutta_richtarik_rls: decompose_problems} by substituting $x_k$ for $x$, which makes $\phi$ dependent on $x$  via $x$ appearing in the numerator only. The problem of minimizing the resulting function in $x$
is a {\em weighted least squares problem}. The normal equations for this problem have the form 
\begin{equation}\label{dutta_richtarik_irls: 1}
Q^\top W_0(x_k) Q x=Q^\top W_0(x_k) b.
 \end{equation}
IRLS is obtained by setting $x_{k+1}$ to be equal to the solution of \eqref{dutta_richtarik_irls: 1}. For stability purposes, however, we shall use weight matrices $W_\delta(x_k)$ for some threshold parameter $\delta>0$ instead. This leads to the IRLS method:
 \begin{equation}\label{dutta_richtarik_irls: equation}
\boxed{x_{k+1}=(Q^\top W_\delta(x_k) Q)^{-1} Q^\top W_\delta (x_k ) b}
 \end{equation}
Osborne \cite{osborne} and more recently~\cite{sigl} performed a comprehensive analysis of the performance of IRLS for $\ell_p$ minimization with  $1<p<3$.

\subsection{Homotopy Method}

In this section we generalize the IRLS method \eqref{dutta_richtarik_irls: equation} by introducing a {\em homotopy}~\cite{Burrus_homotopy} parameter $1\leq p \leq 2$.  We set $p_0=2$ and choose $x_0 \in \R^t$ (in our experiments, random initialization will do). Consider the function
\[
\phi_p(x,y) := \sum_{j=1}^m \frac{(q_j^\top x - b_j)^2}{|q_j^\top y - b_j|^{2-p}}.\]
Note that $\phi_1(x,x) $ is identical to the $\ell_1$ regression function $\phi$ appearing in \eqref{dutta_richtarik_rls: decompose_problems}. Given current iterate $x_k$, consider function $\phi_{p_k}(x,x_k)$. This is a weighted least squares function of $x$. Our homotopy method is defined by setting
\[ x_{k+1} = \arg\min_x\phi_{p_k} (x, x_k),  \]
and  subsequently decreasing the homotopy parameter as $p_{k+1} = \max\{p_k \eta, 1\}$, where $0<\eta< 1$ is a constant reduction factor.


As in the case of IRLS, the normal equations for the above problem have the form 
 \begin{equation}\label{dutta_richtarik_homotopy_matrix: decompose_problems}
Q^\top W_{0,p_k} (x_k) Q x= Q^\top W_{0,p_k} (x_k) b,
 \end{equation}
 where 
$W_{\delta,p}(x):={\rm Diag}(1 / \max \{ |  q_j^\top x - b_j |^{2-p} , \delta \} ).$
The  (stabilized) solution of \eqref{dutta_richtarik_homotopy_matrix: decompose_problems} is given by
\begin{equation} \boxed{x_{k+1} =    (Q^\top W_{\delta,p_k} (x_k) Q)^{-1}  Q^\top W_{0,p_k} (x_k) b } \label{alg:homotopy}\end{equation}

As mentioned above, one step of the homotopy scheme \eqref{alg:homotopy} is identical to one step of IRLS~\eqref{dutta_richtarik_irls: 1} when $p_k=1.$ In practice, however,  the homotopy method sometimes performs better~(see Figures~\ref{compare_methods},~\ref{Toscana}, and Table~\ref{sbi_quant}).  
\begin{figure}
    \centering
    \includegraphics[width = 0.5\textwidth]{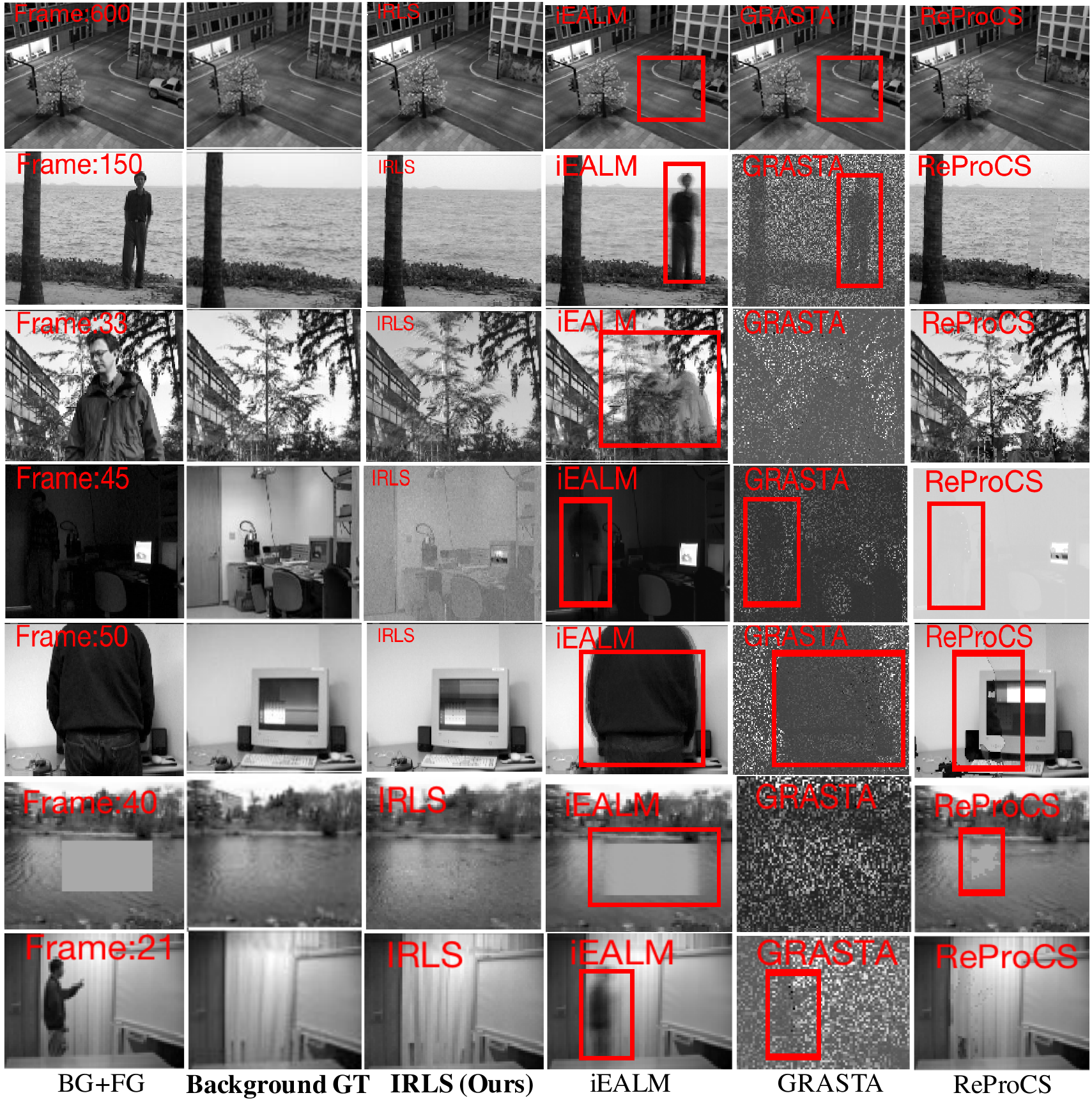}
    \caption{\small{Background recovered on Stuttgart, Wallflower, and I2R dataset.~Comparing with the ground truth~(second column), IRLS recovers the best quality background.}}
    \label{Qual_comparison}
\end{figure}
\begin{figure*}
    \centering
    \includegraphics[width =\textwidth]{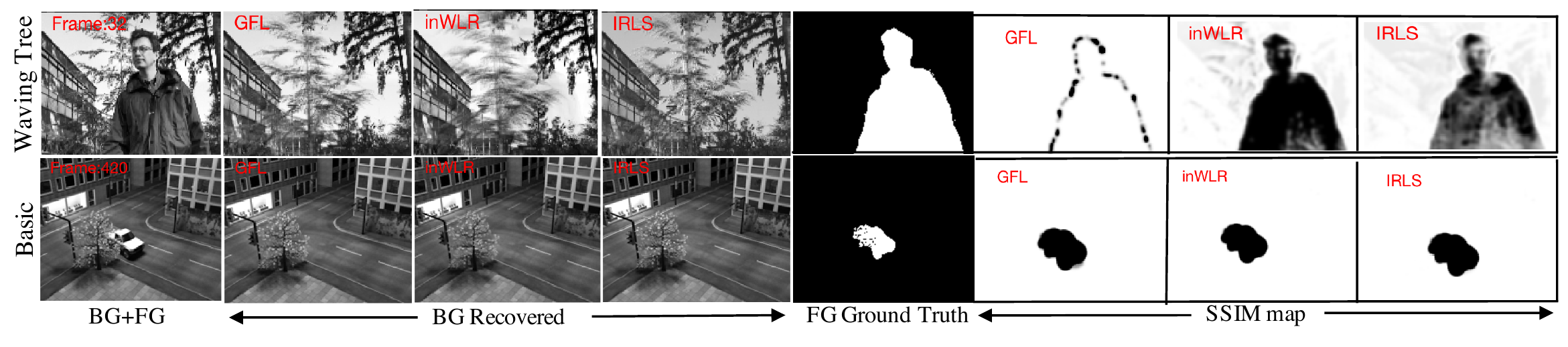}
    \caption{\small{Qualitative and Quantitative comparison with supervised GFL and inWLR.~GFL and IRLS construct better backgrounds on the {\tt Waving Tree} video.~On {\tt Basic} video all the methods have similar performance. However, supervised GFL takes 117.11 seconds and 6.25 seconds on {\tt Waving Tree} and {\tt Basic} video, respectively, to process 1 frame; whereas inWLR takes \textcolor{blue}{3.39} seconds and  \textcolor{blue}{17.83} seconds, respectively on those two sequences.~In contrast,~IRLS takes \textcolor{red}{0.59} seconds and \textcolor{red}{7.02} seconds, respectively and recovers the similar SSIM map.}}
    \label{Qual_comparison:GFL_inWLR}
\end{figure*}
\begin{figure*}
    \centering
    \includegraphics[width = \textwidth]{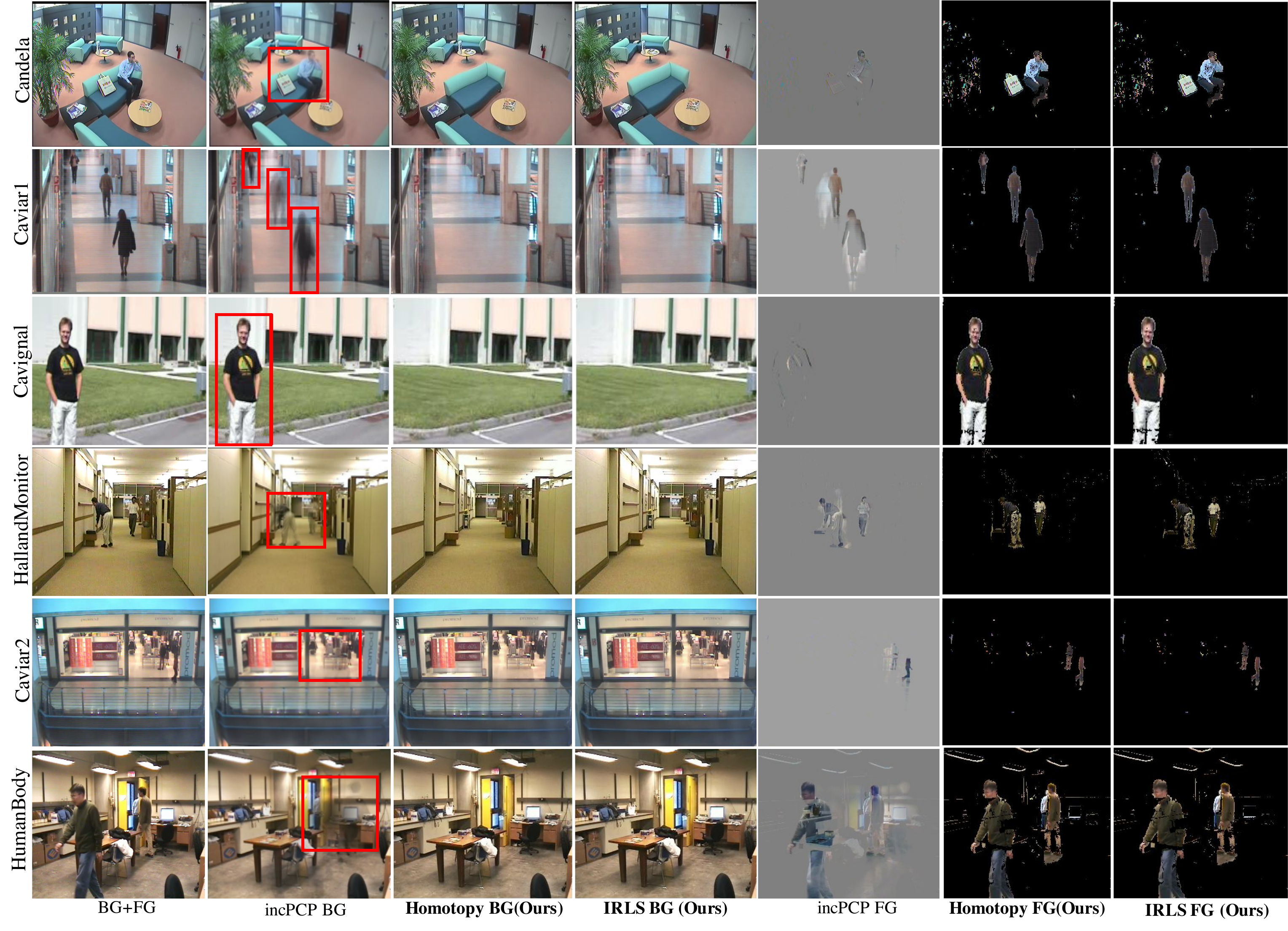}
    \caption{\small{Background and foreground recovered by online methods on SBI dataset.~The videos have static, semi-static foreground, newly added static foreground, shadows that already present in the background and newly created by moving foreground, and occlusion and disocclusion of static and dynamic foreground.~For a comprehensive review of the dataset we refer the readers to~\cite{sbi_a}.}}
    \label{SBI}
\end{figure*}
\begin{figure*}
    \centering
    \includegraphics[width = \textwidth]{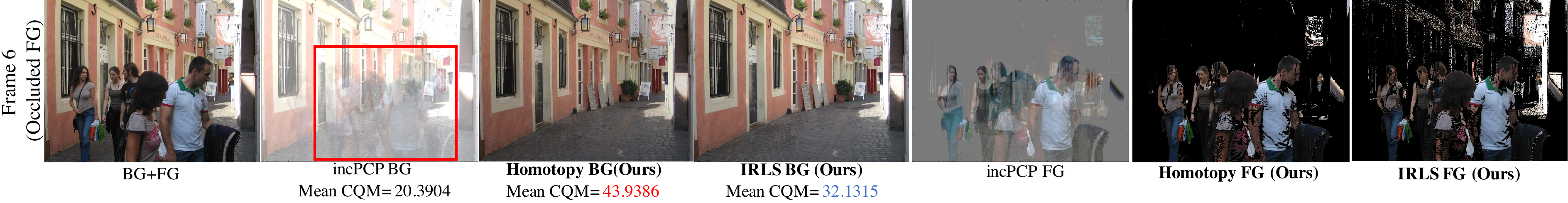}
    \caption{\small{Qualitative and Quantitative comparison on {\tt Toscana-HD} video.~Besides IRLS and Homotopy, the two best methods on {\tt Toscana}, that is, Photomontage~\cite{photomontage} and SOBS1~\cite{SOBS} have MSSIM 0.9616 and 0.9892 and CQM 50.2416 and 43.3002, respectively~\cite{boumans_taxonomy}.}}
    \label{Toscana}
\end{figure*}
\subsection{Stochastic Subgradient Descent}

In this section we propose the use of {\em two variants} of stochastic subgradient descent (SGD) to solve~\eqref{dutta_richtarik_l1: simple}:
\begin{equation} \label{eq:i098sh09s}
\min_{x\in \R^t} \phi(x) := \frac{1}{m}\sum_{j=1}^m \phi_j(x),\end{equation}
where $\phi_j(x) := m |q_j^\top x - b_j|$. Functions $\phi_j$ are convex, but not differentiable. However, they are subdifferentiable. A classical result from convex analysis says that the subdifferential of a sum of convex functions is the sum of the subdifferentials. Therefore, the subddiferential  $\partial \phi$ of $\phi$ is given by the formula
$\partial \phi(x) = \frac{1}{m} \sum_{j=1}^m \partial \phi_j(x).$
In particular, if we choose $j\in [m]$ uniformly at random, and pick  $g_j(x) \in \partial \phi_j(x)$, then ${\rm E} [g_j(x)] \in \partial \phi(x)$. That is, $g_j(x)$ is an unbiased estimator of a subgradient of $\phi$ at $x$.

A generic SGD method applied to \eqref{eq:i098sh09s} (or, equivalently, to \eqref{dutta_richtarik_l1: simple}) has the  form
\begin{equation}
\label{alg:SGD-generic} \boxed{x_{k+1} = x_k - \eta_k g_i(x_k)}
\end{equation}


 An easy calculation using the chain rule for subdifferentials of convex functions gives the following formula for $\partial\phi_j(x)=m q_j \partial | q_j^\top x-b_j|$  (see, for instance, \cite{nesterov:book}):
\begin{eqnarray}\label{partial}
\partial\phi_j(x)&=&\left\{
  \begin{array}{@{}ll@{}}
    m q_j, & \text{if}\ q_j^\top x-b_j >0 \\
    -m q_j, & \text{if}\ q_j^\top x- b_j <0 \\
    0, & \text{otherwise}
  \end{array}.\right.
\end{eqnarray}


When $q_j^\top x_k-b_j$ is nonzero, each iterate of SGD moves in the direction of either vector $q_j$ or $-q_j$, with an appropriate stepsize. The initialization of the method (i.e., choice of $x_0\in \R^t$ and the learning rate parameters $\eta_k$) plays an important role in the convergence of the method. 

We consider two variants of SGD depending on the choice of $\eta_k$ and on the vector that we output. 
\begin{table}
\begin{center}
\begin{tabular}{|c|c|c|}
\hline
method & $\eta_k$ & output \\
\hline
SGD 1 & $\tfrac{R}{\sqrt{k} \|g_j(x_k)\|}$ & $x_k$ \\
\hline
SGD 2 & $\tfrac{B}{\rho \sqrt{K}}$ & $\hat{x}_K = \tfrac{1}{K}\sum_{k=0}^{K-1} x_k$ \\
\hline
\end{tabular}
\end{center}
\caption{\small Two variants of SGD.}
\label{tbl:SGD}
\end{table}
In SGD 1 we always normalize each stochastic subgradient, and multiply the resulting vector by $R/\sqrt{k}$, where $k$ is the iteration counter, for some constant $R>0$ which needs to be tuned. This method is a direct extension of the subgradient descent method in \cite{nesterov:book}. The output is the last iterate. While we provide no theoretical guarantees for this method, it performs well in our experiments.~On the other hand, SGD 2 is a more principled method. This arises as a special case of the SGD mehod described and analyzed in \cite{MLbook}. In this method, one needs to decide on the number of iterations $K$ to be performed in advance. The method ultimately outputs the average of the iterates. The stepsize $\eta_k$ is set to $B/\rho \sqrt{K}$, where $B>0$ and $\rho>0$ are parameters the value of which can be derived from the following iteration complexity result:
\begin{theorem}[\cite{MLbook}] Let  $x_*$ be a solution of \eqref{eq:i098sh09s} and let $B>0$ be such that 
$ \|x_*\|\leq B$. Further, assume that $\|g_j(x)\|\leq \rho$ for all $x\in \R^t$ and $j\in [m]$. If SGD 2 runs for $K$ iterations with $\eta =\frac{B}{\rho\sqrt{T}}$, then 
$
{\rm E}[\phi(\hat{x}_K)]-\phi(x_*)\leq \frac{B\rho}{\sqrt{K}},
$
where $
\hat{x}_K$ is given as in Table~\ref{tbl:SGD}. Moreover, for any $\epsilon>0$ to achieve ${\rm E}[\phi(\hat{x}_K)]-\phi(x_*)\leq \epsilon$ it suffices to run SGD 2 for $K$ iterations where $K\geq \frac{B^2\rho^2}{\epsilon^2}.$
\end{theorem}
\begin{table*}
\footnotesize
\begin{center}
\begin{tabular}{|l|c|c|c|}
\hline
Dataset & Video & No. of frames & Resolution\\
\hline
 Stuttgart~\cite{cvpr11brutzer} & {\tt Basic}~(Grayscale) & 600 & $144\times176$  \\
 &  {\tt Basic}~(RGB-HD) & 600 & $600\times 800$ \\
 &  {\tt Lightswitch}~(RGB-HD) & 600 & $600\times 800$ \\
SBI~\cite{sbi_a}& {\tt IBMTest2}~(RGB) &91& $320\times240$ \\
& {\tt Candela}~(RGB) &351 & $352\times 288$\\
& {\tt Caviar1}~(RGB) & 610&$384\times 288$  \\
& {\tt Caviar2}~(RGB) &461 &$384\times 288$  \\
 &{\tt Cavignal}~(RGB)& 258 & $200\times 136$  \\
 & {\tt HumanBody}~(RGB) & 741 & $320\times 240$ \\
& {\tt HallandMonitor}~(RGB)& 296& $352\times 240$ \\
 & {\tt Highway1}~(RGB) & 440 & $320\times 240$  \\
& {\tt Highway2}~(RGB) & 500 & $320\times 240$  \\
&{\tt Toscana}(RGB-HD) &6 & $600\times 800$ \\
Wallflower~\cite{wallflower} &{\tt Waving Tree}(Grayscale) &66 & $120\times 160$ \\
&{\tt Camouflage} (Grayscale) &52 & $120\times 160$ \\
I2R/Li dataset~\cite{lidata} & {\tt Meeting Room}(Grayscale) &1209 & $64\times 80$ \\
&{\tt Watersurface} (Grayscale) &162 & $128\times 160$ \\
&{\tt Lightswitch}(Grayscale) &1430 & $120\times 160$ \\
&{\tt Lake}(Grayscale) &80 & $72\times 90$ \\
\hline
\end{tabular}
\end{center}
\caption{\small{Data used in this paper.}}\label{dataset}
\end{table*}
\begin{table*}[!h]
\footnotesize
\begin{center}
\begin{tabular}{|l|c|c|c|c|}
\hline
Algorithm & Abbreviation & Appearing in Experiment & Reference\\
\hline
Iterative Reweighted Laast Squares& IRLS & Figure \ref{compare_methods}--\ref{Mean_2},and Table\ref{sbi_quant},\ref{time}& This paper\\
Homotopy & Homotopy& Figure ~\ref{compare_methods},\ref{SBI}--\ref{Mean_2}, and Table \ref{sbi_quant}, \ref{time}& This paper\\
Stochastic Subgradient Descent 1& SGD 1 & Figure \ref{compare_methods} and Table \ref{tbl:SGD}& This paper\\
Stochastic Subgradient Descent 2& SGD 2 & Figure \ref{compare_methods} and Table \ref{tbl:SGD}& This paper\\
Inexact Augmented Lagrange Method of Multipliers& iEALM &Figure \ref{ROC_methods}-\ref{Qual_comparison}&\cite{LinChenMa}  \\
Supervised Generalized Fused Lasso & GFL & Figure \ref{Qual_comparison:GFL_inWLR}, \ref{Stuttgart}& \cite{xin2015} \\
 Grassmannian Robust Adaptive Subspace
Tracking Algorithm & GRASTA&Figure~\ref{ROC_methods}--\ref{Qual_comparison} & \cite{grasta} \\
Recurssive Projected Compressive Sensing & ReProCS &Figure~\ref{ROC_methods}--\ref{Qual_comparison} &\cite{reprocs,pracreprocs,modified_cs}  \\
Incremental Weighted Low-Rank & inWLR&Figure \ref{Qual_comparison:GFL_inWLR}& \cite{inWLR}\\
Incremental Principal Component Pursuit & incPCP & Figure \ref{SBI}--\ref{Mean_2}, and Table \ref{sbi_quant}, \ref{time} &\cite{incpcp,matlab_pcp,inpcp_jitter}\\
Background estimated by weightless neural networks & BEWIS&Table \ref{sbi_quant}&\cite{BEWIS}   \\
Independent Multimodal Background Subtraction Multi-Thread & IMBS-MT &Table \ref{sbi_quant}& \cite{IMBS} \\
RSL2011& - &Table \ref{sbi_quant}& \cite{RSL} \\
Color Median&-&Table \ref{sbi_quant}&\cite{colormedian}\\
Photomontage &-&Table \ref{sbi_quant}& \cite{photomontage} \\
Self-Organizing Background Subtraction1 & SOBS1 &Table \ref{sbi_quant}& \cite{SOBS}\\
\hline
\end{tabular}
\end{center}
\caption{\small{Algorithms compared in this paper.}}\label{algo}
\end{table*}

\section{Numerical Experiments}
To validate the robustness of our proposed algorithms, we tested them on some challenging real world and synthetic video sequences containing occlusion, dynamic background, static, and semi-static foreground.~For this purpose, we extensively use 19 gray scale and RGB videos from the Stuttgart, I2R, Wallflower, and the SBI dataset~\cite{cvpr11brutzer,lidata,sbi_a, SBI_web, wallflower}. .~We refer the readers to Table \ref{dataset} to get an overall idea of the number of frames of each video sequence used, video type, and resolution. 

For quantitative measure, we use the receiver operating characteristic~(ROC) curve, recall and precision~(RP) curve, the structural similarity index (SSIM), SSIM map~\cite{mssim}, multi-scale structural similarity index (MSSSIM)~\cite{msssim}, and~color image quality measure~(CQM)~\cite{boumans_taxonomy,cqm}.~Due to the availability of ground truth~(GT) frames, we use the Stuttgart artificial dataset~(has foreground GT) and the SBI dataset~(have background GTs) to analyze the results quantitatively and qualitatively.~To calculate the average computational time we ran each algorithm five times on the same dataset and compute the average.
Throughout this section, the best and the $2^{\rm nd}$ best results are colored with \textcolor{red}{red} and \textcolor{blue}{blue}, respectively.

\subsection{Comparison between our proposed algorithms}
First we compare the performance of our proposed algorithms in batch mode on the {\tt Basic} scenario.~Figure~\ref{compare_methods} shows that all four algorithms are are very competitive and we note that IRLS has the least computational time.~We ran each of IRLS and Homotopy method for five iterations, and SGD 1 and SGD 2 for 5000 iterations.~IRLS takes \textcolor{red}{7.02} seconds, Homotopy takes \textcolor{blue}{8.47} seconds, SGD 1 takes 17.81 seconds, and SGD 2 takes 17.67 seconds.~We mention that the choice of $R$ in SGD 1 and $B$ and $\rho$ in SGD 2 are problem specific.~Due to computational efficiency, we compare IRLS $\ell_1$ with other batch methods in the next section.
\begin{table*}
\footnotesize
\begin{center}
\begin{tabular}{|l|c|c|c|c|c|c|c|c|}
\hline
Video & SOBS1  &RSL2011& IMBS-MT &BEWIS & Color Median &IRLS &Homotopy & incPCP\\
\hline
 {\tt IBMTest2}&\textcolor{red}{\bf 0.9954} &0.9303&0.9721&0.9602& 0.9939&0.9950 &\textcolor{blue}{\bf 0.9953}&0.9670\\
{\tt Candela} &0.9775 &{0.9916} & 0.9893 & 0.9852& 0.9382& \textcolor{red}{\bf 0.9995} & \textcolor{blue}{\bf 0.9992} &0.9412\\
{\tt Caviar1}& 0.9781 &0.9947 &0.9967 &0.9813& 0.9918& \textcolor{red}{\bf 0.9994} & \textcolor{blue}{\bf 0.9993}&0.8649\\
{\tt Caviar2} & 0.9994 &0.9962  &0.9986  &0.9994 & 0.9994 & \textcolor{red}{\bf 0.9999} & \textcolor{blue}{\bf 0.9998} &0.9935\\
{\tt Cavignal} &0.9947 &0.9973 &0.9982 & \textcolor{blue}{\bf 0.9984} & 0.7984 & \textcolor{red}{\bf 0.9989} &0.9975 &0.8312\\
{\tt HumanBody}&0.9980 &0.9959  &0.9958 &0.9866 & 0.9970 & \textcolor{red}{\bf 0.9996} & \textcolor{blue}{\bf 0.9990} &0.9360\\
{\tt HallandMonitor} &0.9832 &0.9377   &0.9954 &0.9626    &0.9640     &  \textcolor{blue}{\bf 0.9991} &  \textcolor{red}{\bf 0.9992}     &0.9355\\
{\tt Highway1} &0.9968 &0.9899      &0.9939     &0.9886       & 0.9924   & \textcolor{blue}{\bf 0.9980}      & \textcolor{red}{\bf 0.9985}     &0.8847\\
{\tt Highway2 }&0.9991 &0.9907    &0.9960       &0.9942  & 0.9961    &\textcolor{blue}{\bf 0.9994}  &\textcolor{red}{\bf 0.9997}     &0.9819\\
{\tt Toscana} & 0.9616 & 0.0662 & 0.8903 & 0.8878 &0.8707 & \textcolor{blue}{\bf  0.9853} & \textcolor{red}{\bf 0.9996} & 0.8416\\
\hline
{\bf Average} &0.9814 &0.9491   &0.9929       &0.9745  & 0.9542   &\textcolor{blue}{\bf 0.9975}  &\textcolor{red}{\bf 0.9987}     &0.9177\\
\hline
\end{tabular}
\end{center}
\caption{\small{Comparison of average MSSSIM of the different methods on SBI dataset. Source:~\cite{SBI_web, boumans_taxonomy,sbi_a}.}}\label{sbi_quant}
\end{table*}

\subsection{Comparison with RPCA,~GFL, and other state-of-the-art methods}

In this section we compare IRLS with other state-of-the-art batch background estimation methods, such as, iEALM~\cite{LinChenMa} of RPCA, GRASTA, and ReProCS on the {\tt Basic} scenario.~Figure~\ref{ROC_methods} shows that IRLS sweeps the maximum area under the ROC curve.~Additionally, in Figure~\ref{ssim} IRLS has the best mean SSIM~(MSSIM) among all other methods.~Moreover, in batch mode, IRLS takes the least computational time.

Next in Figure~\ref{Qual_comparison} we present the background recovered by each method on Stuttgart, Wallflower, and I2R dataset. The video sequences have occlusion, dynamic background, and static foreground. IRLS can detect the static foreground and also robust to sudden illumination changes. 
\begin{figure*}
    \centering
    \includegraphics[width = \textwidth]{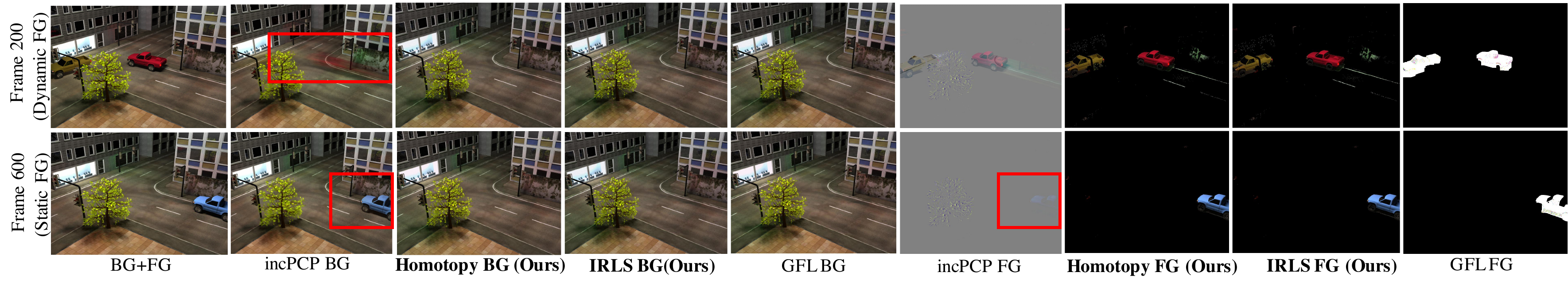}
    \caption{\small{Qualitative comparison on {\tt Basic} scenario HD scene.~The {\bf SSIM}s are~(the $1^{\rm st}$ number indicates frame 200 and the $2^{\rm nd}$ number indicates frame 600):~incPCP 0.021 and 0.0173, IRLS 0.9089 and \textcolor{red}{0.9731}, homotopy \textcolor{blue}{0.9327} and \textcolor{blue}{0.9705}, GFL \textcolor{red}{0.9705} and 0.9310.~The {\bf MSSSIM}s are:~incPCP 0.6315 and 0.4208, IRLS 0.8777 and \textcolor{red}{0.9746}, homotopy \textcolor{blue}{0.9166} and \textcolor{blue}{0.9725}, GFL \textcolor{red}{0.9175} and 0.9645.}}
    \label{Stuttgart}
\end{figure*}
\begin{figure*}
    \centering
    \includegraphics[width = \textwidth]{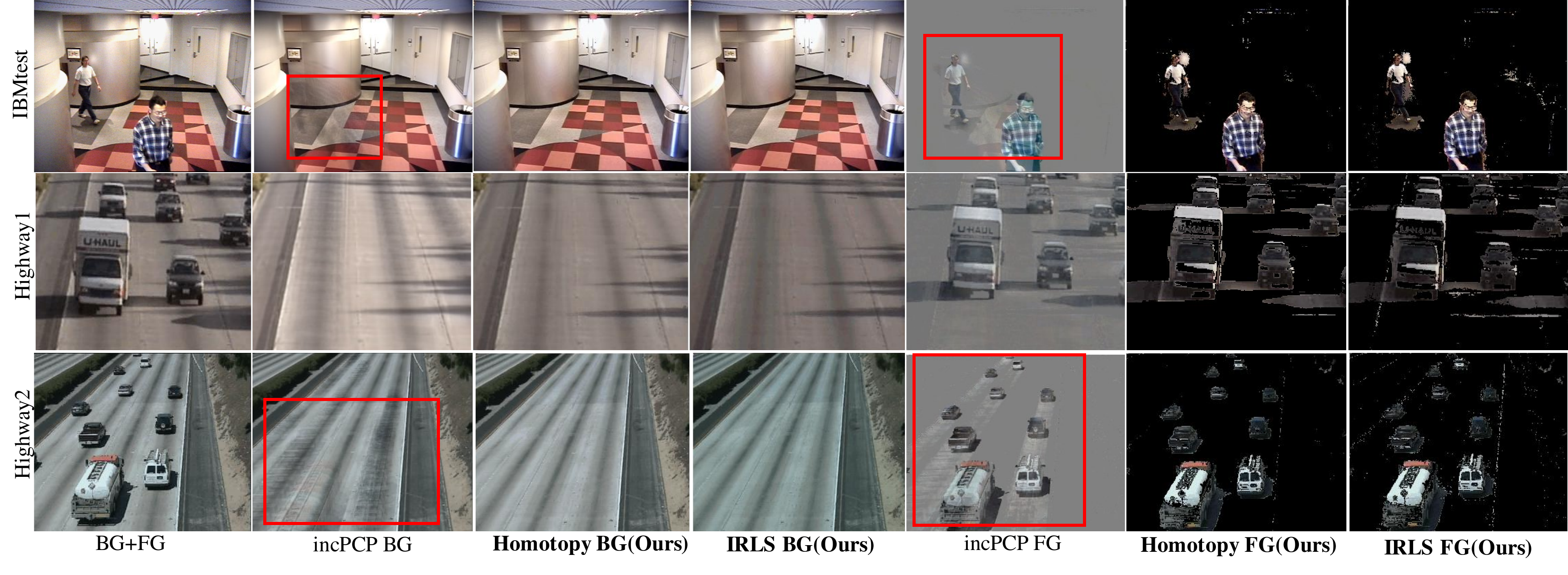}
    \caption{\small{Background and foreground recovered by online methods on SBI dataset. The videos have shadows that already present in the background and newly created by moving foreground, occlusion and disocclusion of dynamic foreground.}}
    \label{SBI_2}
\end{figure*}
\begin{figure*}
    \centering
    \includegraphics[width = \textwidth]{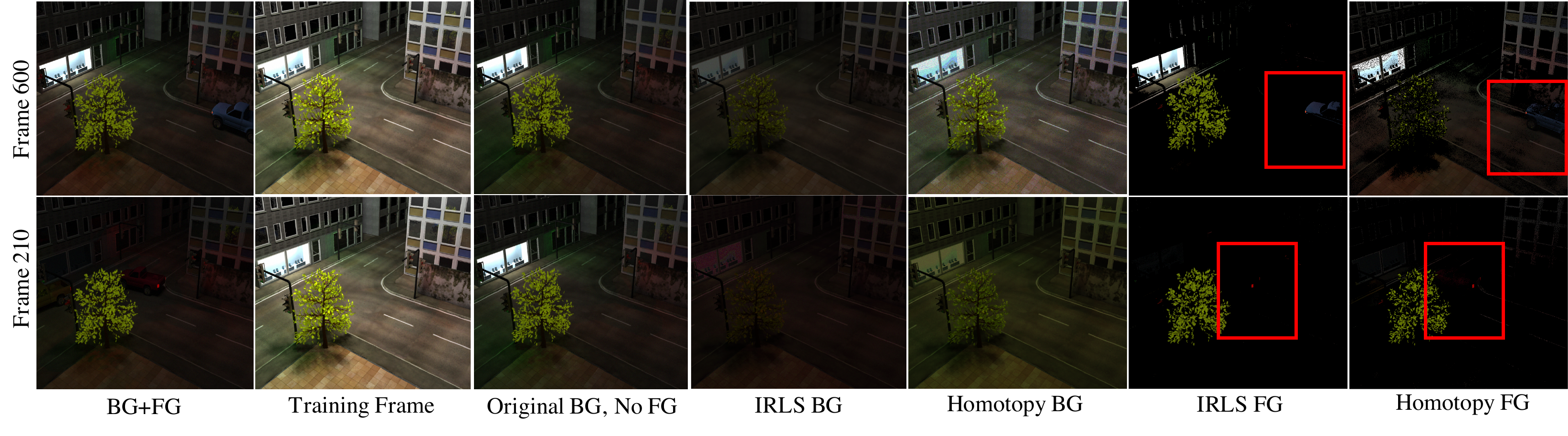}
    \caption{\small{Background and foreground recovered by our proposed online methods on {\tt Lightswitch} video. Both IRLS and homotopy captures the effect of change in illumination, irregular movements of the tree leaves, and reflections. Comparing with the No FG image both of our proposed method do pretty well.}}
    \label{LS}
\end{figure*}
\begin{figure}
\centering
\includegraphics[width=0.45\textwidth]{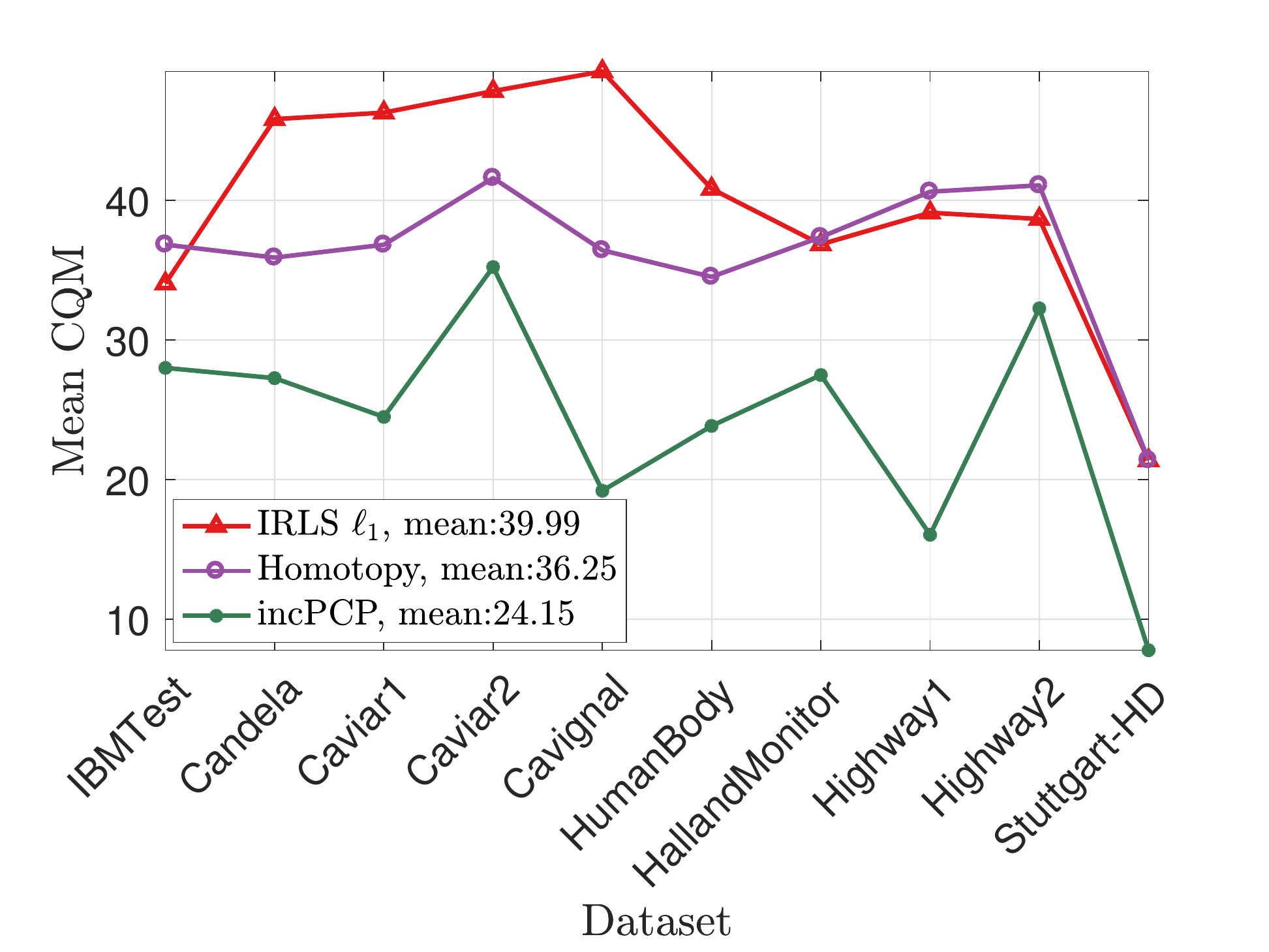}%
\caption{\small{Mean CQM of online methods on SBI dataset and {\tt Basic-HD} video.~The higher the CQM value, the better is the recovered image.}}%
 \label{Mean_2}
\end{figure}

Finally, we compare our IRLS with the supervised GFL model of Xin \etal~\cite{xin2015} and inWLR of Dutta \etal.~\cite{inWLR}~(see Figure~\ref{Qual_comparison:GFL_inWLR}).~For {\tt Waving Tree} scenario, supervised GFL uses 200 training frames and it takes 117.11 seconds to compute the background and foreground from one training frame and the ssim of the FG is \textcolor{red}{0.9996}.~inWLR does not use any training frames and takes \textcolor{blue}{3.39} seconds to compute the background and foreground from the video sequence that consists of 66 frames and the MSSIM is \textcolor{blue}{0.9592}. In contrast, IRLS uses 15 training frames and takes \textcolor{red}{0.59} seconds to process the entire video with an MSSIM 0.9398. For {\tt Basic} scenario, supervised GFL again uses 200 training frames and  takes 6.25 seconds to process one training frame and the ssim of the FG is {0.9462}.~inWLR does not require any training frame and takes \textcolor{blue}{17.83} seconds to process 600 frames in a batch-incremental mode and the MSSIM is \textcolor{blue}{0.9463}. In contrast, IRLS uses only 15 training frames and takes \textcolor{red}{7.02} seconds to process the entire video and the MSSIM is \textcolor{red}{0.9524}.

\subsection{Online implementation on RGB videos}
In this section we show the robustness of two of our algorithms on RGB videos in online mode.~Due to the space limitation we only provide results on IRLS and homotopy algorithm~(these two methods were also the fastest in the batch mode).~Primarily, we compare our results with incPCP and GFL~\cite{incpcp,matlab_pcp,inpcp_jitter,xin2015}.~We should mention that besides incPCP, probabilistic robust matrix factorization~(PRMF)~\cite{prmf} and RPCA bilinear projection~(RPCA-BL)~\cite{RPCA-BL} has online extensions.~However, PRMF uses the entire available data in its batch normalization step and there is no available implementation of online RPCA-BL.~To the best of our knowledge incPCP is the only state-of-the-art online method which deals with HD RGB videos in full online mode.~The incPCP code is downloaded from the author's website\footnote{https://sites.google.com/a/istec.net/prodrig/Home/en/pubs/incpcp}.~As mentioned in the software package we use the standard PCP~(fixed camera) mode for incPCP~\cite{incpcp,matlab_pcp} implementation.

\paragraph{Discussions.}We use {\tt Basic-HD} and the SBI dataset to provide extensive qualitative and quantitative comparison. The online mode of our algorithm only uses the available pure background frames to learn the basis $Q$ for each color channel and then operate on each test frame in a complete online mode.~Note that we only use 10 training frames and we strongly believe that one can use even less number of training frames to obtain almost the similar performance.~Homotopy uses less iterations than IRLS to produce a comparable background and hence it is faster than IRLS in online mode.~In Figure~\ref{SBI} and \ref{SBI_2} we compare IRLS and homotopy against incPCP on the SBI dataset.~Compare to the ghosting appearances in the incPCP backgrounds, our online methods construct a clean background for each video sequence.~We also removed the static foreground, occluded foreground, and the foreground shadows.~In Figure~\ref{Toscana} and~\ref{Stuttgart} we show our performance on HD video sequences.~In addition to incPCP, we compared with supervised GFL on the {\tt Basic-HD}~(see Figure~\ref{Stuttgart}).~Supervised GFL uses 200 training frames~(the average processing time of the training frames is 7.31 seconds) and takes 431.78 seconds to process each test frame and produce a comparable quantitative result as online IRLS and homotopy.~For computational time comparison with incPCP we refer to Table~\ref{time}.
Finally we provide the results of online IRLS and homotopy on one of the most challenging HD video sequences, that is, {\tt Lighswitch} of the Stuttgart dataset. This scenario is a nighttime scenario and has varying illumination effects throughout the video sequence. Starting from frame 125 the illumination suddenly changes. Additionally, it has reflections, traffic light change, and movements of the tree leaves. We used 10 daytime pure background frames for training purpose and by using them we estimated the nighttime scene. As expected in Figure~\ref{LS} both IRLS and homotopy perform pretty well with the changing illumination which can be verified from the pure {\tt Lightswitch} BG frame (Figure~\ref{LS} third column). 
Additionally, we compare our quantitative results against other state-of-the-art algorithms, such as, the adaptive neural background algorithm aka Self-Organizing Background Subtraction1~(SOBS1)~\cite{SOBS}, Photomontage~\cite{photomontage}, Color Median,~RSL2011~\cite{RSL}, Independent Multimodal Background Subtraction Multi-Thread~(IMBS-MT)~\cite{IMBS}, background estimated by weightless neural networks~(BEWIS)~\cite{BEWIS} on SBI dataset.~We refer~Table~\ref{sbi_quant}~(Source:~\cite{boumans_taxonomy}) and~Figure~\ref{Toscana}.~Finally in Figure~\ref{Mean_2} we provide the mean CQM of the online methods on SBI dataset and {\tt Basic-HD} video. In online mode, IRLS and Homotopy outperform incPCP in mean CQM and mean MSSSIM in each video. 
\section{Conclusion}
We proposed a novel and fast model for supervised video background estimation. Moreover, it is robust to several background estimation challenges.~We used the simple and well-known $\ell_1$ regression technique and provided several online and batch background estimation methods that can process high resolution videos accurately. Our extensive qualitative and quantitative comparison on real and synthetic video sequences demonstrated that our supervised model outperforms the state-of-the-art online and batch methods in almost all cases.

\section{Appendix 1: Historical Comments}

We start by making a connection between the supervised GFL model proposed by Xin \etal~\cite{xin2015} and the constrained low-rank approximation problem of Golub \etal~\cite{golub}.
\subsection{Golub's constrained low-rank approximation problem}
In 1987, Golub et al.~\cite{golub} formulated the following constrained low-rank approximation problem: Given $A=[A_1\;A_2]\in\mathbb{R}^{m\times n'}$ with $A_1\in\mathbb{R}^{m\times r}$ and $A_2\in\mathbb{R}^{m\times n}$, find $A_G=[\tilde{B}_1\;\tilde{B}_2]$ such that, $\tilde{B}_1 \in\mathbb{R}^{m\times r},\tilde{B}_2\in\mathbb{R}^{m\times n}$, solve:
\begin{eqnarray}\label{golub's problem}
[\tilde{B}_1\;\tilde{B}_2]=\arg\min_{\substack{B=[B_1\;B_2]\\B_1=A_1\\{\rm rank}(B)\le r}}\|A-B\|_F^2,
\end{eqnarray}
where $\|\cdot\|_F$ denotes the Frobenius norm of matrices.~Motivated by \cite{golub}, Dutta \etal recently proposed more general weighted low-rank~(WLR) approximation problems and showed their application in the background estimation problem \cite{duttali_acl,duttali,duttali_bg}. 

\paragraph{Connection with \eqref{golub's problem}.}
Recall that the background estimation model the generalized fused lasso~(GFL) proposed by Xin \etal~\cite{xin2015} with the choice $f_{\rm rank}(B) = {\rm rank}(B)$ and $f_{\rm spar}(F) = \lambda \|F\|_{\rm GFL}$ can be written as:  
\begin{align*}
\min_{B}{\rm rank}(B)+\lambda\|A-B\|_{\rm GFL}.
\end{align*}
In this model, $\|\cdot\|_{GFL}$ is the ``generalized fused lasso'' norm. With the extra assumption that ${\rm rank}(B)={\rm rank}(B_1)$ and by using the $\|\cdot\|_{\rm GFL}$ norm, problem~\eqref{gfl} is a constrained low rank approximation problem as in~\eqref{golub's problem} and can be written as follows:
\begin{equation*}\label{gfl_golub}
\boxed{\min_{B = [B_1\;B_2]}\{\|A-B\|_{\rm GFL} \; \text{subject to}\; {\rm rank}(B)\le r, B_1=A_1\}.}
\end{equation*}

\begin{table}
\footnotesize
\begin{center}
\begin{tabular}{|l|c|c|c|c|}
\hline
Video~(No. of frames) & IRLS & Homotopy & incPCP\\
\hline
 {\tt IBMTest2}~(91)& 37.28 & \textcolor{red}{\bf 21.84} &22.45  \\
{\tt Candela}~(351) &163.80  &\textcolor{blue}{\bf 133.6} & \textcolor{red}{\bf 72.15} \\
{\tt  Caviar1}~(610)& 279.99 & \textcolor{blue}{\bf 213.99} & \textcolor{red}{\bf 120.58} \\
{\tt  Caviar2}~(461)&199.16  & \textcolor{blue}{\bf 158.1} & \textcolor{red}{\bf 85.68}  \\
{\tt Cavignal}~(258)& 71.26 & \textcolor{blue}{\bf 70.77}& \textcolor{red}{\bf 39} \\
{\tt  HumanBody}~(741)&261.94  & \textcolor{blue}{\bf227.25} & \textcolor{red}{\bf 134.83}\\
{\tt  HallandMonitor}~(296)&116.86  & \textcolor{blue}{\bf 88.99} & \textcolor{red}{\bf 59.63}\\
{\tt Highway1}~(440)&155.84  &\textcolor{blue}{\bf 134.03} & \textcolor{red}{\bf 81.44}\\
{\tt  Highway2}~(500)&181.85  &\textcolor{blue}{\bf 156.92} & \textcolor{red}{\bf 87} \\
{\tt Basic-HD}~(600)&599.06  & \textcolor{blue}{\bf 457.2464} & \textcolor{red}{\bf 382.41}\\
{\tt  Toscana-HD}~(6)&7.73  & \textcolor{blue}{\bf 5.13} & \textcolor{red}{\bf 3.1}\\
\hline
\end{tabular}
\end{center}
\caption{\small{Computational time~(in seconds) comparison for online methods.}}\label{time}
\end{table}

\section{Appendix 2 : Augmented Lagrangian Method of Multipliers~(ALM)}

In this section, we demonstrate  an additional background estimation method by using the decomposition model used in the main paper. This method was not described in the main paper. As mentioned in the {\bf Further contributions} Section, we device a batch background estimation model (\textcolor{red}{fifth method})  by using the augmented Lagrangian method of multipliers (ALM). 

 \subsection{The algorithm}
 The Augmented Lagrangian method of multipliers are one of the most popular class of algorithms in convex programming.~In our setup, the proposed method does not provide an incremental algorithm. Instead it relies on fast batch processing of the video sequence.~We can write~\eqref{dutta_richtarik_l1} as an equality constrained problem by introducing the variable $F_2$ as follows:
  \begin{eqnarray}\label{dutta_richtarik_l1_constrained}
&\min_{F_2, S}\|F_2\|_{\ell_1}\nonumber\\
 &{\rm subject~to~} A_2=QS+F_2.
\end{eqnarray}
 We now form the augmented Lagrangian  of~\eqref{dutta_richtarik_l1_constrained}:
\begin{eqnarray}\label{lagrange}
L(S,F_2,Y,\mu) = \|F_2\|_{\ell_1}+\langle Y, A_2-QS-F_2\rangle\\
 +\frac{\mu}{2}\|A_2-QS-F_2\|_F^2,\nonumber
\end{eqnarray}
where $Y\in\mathbb{R}^{m\times n}$ is the Lagrange multiplier, $\langle Y, X \rangle = {\rm Trace} (Y^\top X)$ is the trace inner product, and $\mu>0$ is a penalty parameter. Completing the square and keeping only the relevant terms in~(\ref{lagrange}), for the given iterates $\{S^{(k)},F_2^{(k)},Y^{(k)},\mu_k\}$~we have
\begin{eqnarray*}\label{lagrange1}
S^{(k+1)}&=&\arg\min_{S}L(S,F_2^{(k)},Y^{(k)},\mu_k)\\
&=&\arg\min_S\frac{\mu_k}{2}\left\|A_2-QS-F_2^{(k)}+\frac{1}{\mu_k}Y^{(k)}\right\|_F^2,\\
F_2^{(k+1)}&=&\arg\min_{F_2}L(S^{(k+1)},F_2,Y^{(k)},\mu_k)\\
&=&\arg\min_{F_2} \|F_2\|_{\ell_1}+\frac{\mu_k}{2}\left\|A_2-QS^{(k+1)}-F_2+\frac{1}{\mu_k}Y^{(k)}\right\|_F^2.
\end{eqnarray*}

\begin{algorithm}
\small{
    \SetAlgoLined
    \SetKwInOut{Input}{Input}
    \SetKwInOut{Output}{Output}
    \SetKwInOut{Init}{Initialize}
    \nl\Input{$A=[A_1\;\;A_2] \in\mathbb{R}^{m\times n'}$ (data matrix),~threshold $\epsilon>0,\rho>1, \mu_0>0$\;}
    \nl\Init {$A_1=QR, Y^{(0)}= A_2/{\|A_2\|_{\infty}}, S^{(0)}, F_2^{(0)}$\;}
    \nl \While{not converged}
    {
        \nl $S^{(k+1)}= Q^\top (A_2-F_2^{(k)}+\frac{1}{\mu_k}Y^{(k)})$\;
        \nl $F_2^{(k+1)} = \mathcal{S}_{\frac{1}{\mu_k}}(A_2-QS^{(k+1)}+\frac{1}{\mu_k}Y^{(k)})$\;
        \nl $Y^{(k+1)}=Y^{(k)} +\mu_k(A_2-QS^{(k+1)}-F_2^{(k+1)})$\;
        \nl $\mu_{k+1}=\rho\mu_k$\;
        \nl $k=k+1$\;
    }
    \nl \Output{$S^{(k)},F_2^{(k)}$}}
    \caption{ALM}\label{slm}
\end{algorithm}

The solution to the first subproblem is obtained by setting the gradient of $L(S,F_2^{(k)},Y^{(k)},\mu_k)$ with respect to $S$ to 0, and using the fact that $Q^\top Q = I$:
\begin{eqnarray}\label{sk}
S^{(k+1)}= Q^\top \left(A_2-F_2^{(k)}+\frac{1}{\mu_k}Y^{(k)}\right).
\end{eqnarray}
The second subproblem is the classic sparse recovery problem and its solution is given by 
\begin{eqnarray}\label{fk}
F_2^{(k+1)} = \mathcal{S}_{\frac{1}{\mu_k}}\left(A_2-QS^{(k+1)}+\frac{1}{\mu_k}Y^{(k)}\right),
\end{eqnarray}
where $\mathcal{S}_{\frac{1}{\mu_k}}(\cdot)$ is the elementwise shrinkage function~\cite{shrinkage,shrinkage2}.~We update $Y_k$ and $\mu_k$ via:
\begin{eqnarray}\label{yk}
\left\{\begin{array}{ll}
Y^{(k+1)}=Y^{(k)} +\mu_k(A_2-QS^{(k+1)}-F_2^{(k+1)})\\
\mu_{k+1}=\rho\mu_k
\end{array},\right.
\end{eqnarray}
for a fixed $\rho>1.$

\paragraph{Dual problem.} Next we formulate the Lagrangian dual of~\eqref{dutta_richtarik_l1} to get an insight into the choice of the Lagrange multiplier $Y$. Using standard arguments, we obtain
\begin{eqnarray}\label{lagrange_dual}
\min_{F_2,S \;:\; A_2 = QS + F_2} \|F_2\|_{\ell_1} &=& \min_{F_2,S} \sup_{Y} \|F_2\|_{\ell_1}\notag \\
&&+\langle Y, A_2 - QS - F_2 \rangle \notag \\
& \geq & \sup_{Y}\min_{F_2,S} \|F_2\|_{\ell_1} \notag \\
&&+\langle Y, A_2-QS-F_2\rangle \nonumber\\
&=& \sup_{Y \;:\; \|Y\|_{\infty}\le 1, \; Q^\top Y = 0} \langle Y, A_2\rangle. \notag\\
\end{eqnarray}
The last problem above is the dual of~\eqref{dutta_richtarik_l1} . Clearly, the dual is a linear program. Note that the constraint $Q^\top Y$ dictates that the columns of $Y$ be orthogonal to all columns of $Q$ . That is, the columns of $Y$ must be from the nullspace of $Q$. If we relax this constraint, the resulting problem has a simple closed form solution, namely
\[Y^{(0)} = A_2/{\|A_2\|_{\infty}}.\]
This is a good choice for the initial value of $Y$ in Algorithm~\ref{slm} .

\subsection{Grassmannian robust adaptive subspace estimation~(GRASTA)}\label{Grasta:sec}

Due to close connection with our ALM, we explain the Grassmannian robust adaptive subspace estimation~(GRASTA) in this section. 
In 2012, He et al.~\cite{grasta} proposed GRASTA, a robust subspace tracking algorithm, and showed its application in background estimation problem. Unlike Robust PCA~\cite{LinChenMa,APG}, GRASTA is not a batch-video background estimation algorithm. GRASTA solves the background estimation problem in an incremental manner, considering one frame at a time.~At each time step $i$, it observes a subsampled video frame $a_{{i}_{\Omega_s}}$. That is, each video frame $a_i\in\mathbb{R}^{m}$ is subsampled over the index set $\Omega_s\subset\{1,2,\cdots,m\}$ to produce $a_{{i}_{\Omega_s}}$, where $s$ is the subsample percentage. Similarly, denote the foreground as $F_2 = (f_1,\dots,f_n)$. Therefore, $f_{{i}_{\Omega_s}}\in\mathbb{R}^{|{\Omega_s}|}$ is  a vector whose entries are indexed by $\Omega_s$.~Considering each video frame $a_{{i}_{\Omega_s}}$ has a low rank (say, $r$) and sparse structure, GRASTA models the video frame as:
\begin{eqnarray*}
a_{{i}_{\Omega_s}}=U_{{\Omega_s}}x+f_{{i}_{\Omega_s}}+\epsilon_{{\Omega_s}},
\end{eqnarray*}
where $U\in\mathbb{R}^{{m}\times r}$ be an orthonormal basis of the low-dimensional subspace, $x\in\mathbb{R}^r$ is a weight vector, and $\epsilon_{\Omega_s}\in\mathbb{R}^{|{\Omega_s}|}$ is a Gaussian noise vector. The matrix $U_{\Omega_s}\in\mathbb{R}^{|{\Omega_s}|\times r}$ results from choosing the rows of $U$ corresponding to the index set ${\Omega_s}$.~With the notations above, at each time step $i$, GRASTA solves the following optimization problem:~For a given orthonormal basis $U_{\Omega_s}\in\mathbb{R}^{|{\Omega_s}|\times r}$ solve
\begin{eqnarray}\label{Grasta}
\min_x\|U_{\Omega_s}x-a_{{i}_{\Omega_s}}\|_{\ell_1}.
\end{eqnarray}
Problem~\eqref{Grasta} is the classic least absolute deviations problem similar to~\eqref{dutta_richtarik_l1: decompose} and can be rewritten as:
\begin{eqnarray}\label{Grasta_alm}
&\min_{f_{{i}_{\Omega_s}}}\|f_{{i}_{\Omega_s}}\|_{\ell_1}\nonumber\\
&{\rm subject~to~} U_{{\Omega_s}}x+f_{{i}_{\Omega_s}}-a_{{i}_{\Omega_s}}=0.
\end{eqnarray}
Problem~\eqref{Grasta_alm} can be solved by the use of the augmented Lagrangian multiplier method~(ALM)~\cite{boyd}.~In GRASTA, after updating $x$ and $f_{{i}_{\Omega_s}}$, one has to update the orthonormal basis $U_{\Omega_s}$ as well. The rank one $U_{\Omega_s}$ update step is done first by finding a gradient of the augmented Lagrange dual of~\eqref{Grasta_alm}, and then by using the classic gradient descent algorithm. In summary, at each time step $i$, given a $U^{(i)}\in\mathbb{R}^{{m}\times r}$ and $\Omega_s\subset\{1,2,\cdots,m\}$, GRASTA finds $x$ and $f_{{i}_{\Omega_s}}$ via \eqref{Grasta_alm} and then updates $U_{{\Omega_s}}^{(i+1)}$. This process continues until the video frames are exhausted. 

\paragraph{Comparison between ALM and GRASTA.}
1. At each step of GRASTA, the background and the sparse foreground are given as $U_{{\Omega_s}}x$ and  $a_{{i}_{\Omega_s}} -U_{{\Omega_s}}x$, respectively and then one has to update the basis $U_{{\Omega_s}}$. In contrast,~\eqref{dutta_richtarik_l1_constrained} solves a supervised batch video background estimation problem. In our model, once we obtain the basis set from the $QR$ decomposition of the background matrix $A_1$, we do not update the basis further. 2. GRASTA lacks a convergence analysis which is harder to obtain as the objective function~\eqref{Grasta} in their set-up is only convex in each component.~\cite{grasta}. Our objective function in~\eqref{dutta_richtarik_l1} and in \eqref{lagrange} are convex and therefore allow us to propose a thorough convergence analysis for ALM. 
 
\subsection{Cost of One Iteration}\label{complexity:ALM}

We discuss the complexity of one iteration of Algorithm~\ref{slm} when $A_1$ is of full rank, that is, ${\rm rank}(A_1)=r$. The complexity of the $QR$ decomposition at the initialization step is $\mathcal{O}(2mr^2-\frac{2}{3}r^3)$. Because $r\le r_{{\rm max}}$, the maximum number of available training frames, the above cost can be controlled by the user. Next, the complexity of one iteration of Algorithm~\ref{slm} is $\mathcal{O}(mnr).$ In contrast, the cost of each iteration of GRASTA is $\mathcal{O}(|{\Omega_s}|r^3+Kr|{\Omega_s}|+mr^2),$ where $K$ is the number of inner iterations and $|{\Omega_s}|$ is the cardinality of the index set $\Omega_s\subset\{1,2,\cdots,m\}$  from which each video frame $a_i\in\mathbb{R}^{m}$ is subsampled at a percentage $s$~(see Section~\ref{Grasta:sec}). 

\subsection{Stopping Criteria}
Define $L_{k}:=L(S^{(k)},F_2^{(k)},Y^{(k-1)},\mu_{k-1})$. With the notations above, for a given $\epsilon>0$, Algorithm~\ref{slm} converges if $\|A_2-QS^{(k)}-F_2^{(k)}\|_F/\|A_2\|_F<\epsilon$, or $|L_{k}-L_{k-1}|<\epsilon$, or if the maximum iteration is reached.
\begin{figure*}
    \centering
    \begin{subfigure}{0.48\textwidth}
    \includegraphics[width=\textwidth]{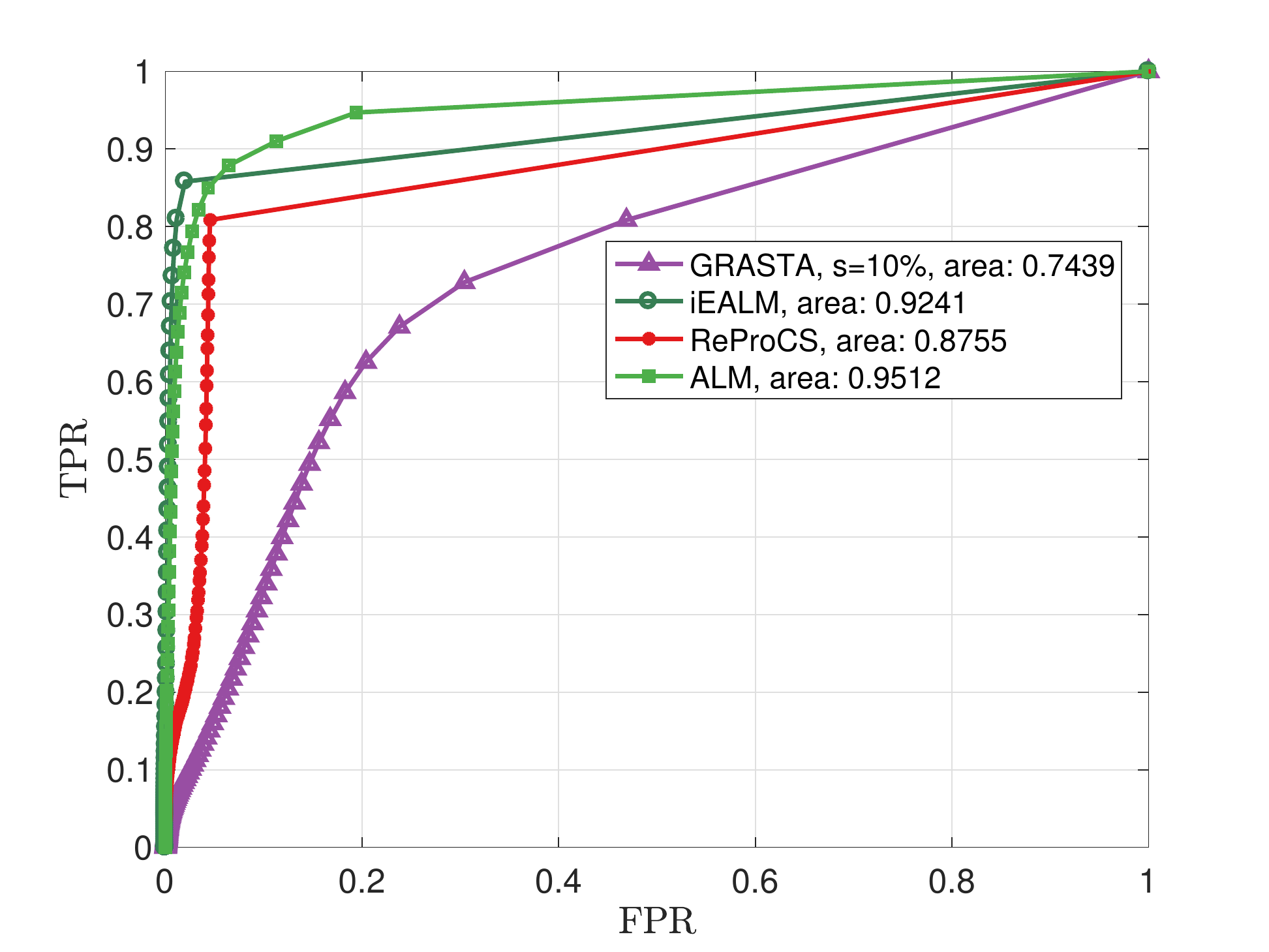}
        \caption{}
      \label{ROC_methods}
      \end{subfigure}
    \begin{subfigure}{0.49\textwidth}
    \includegraphics[width = \textwidth]{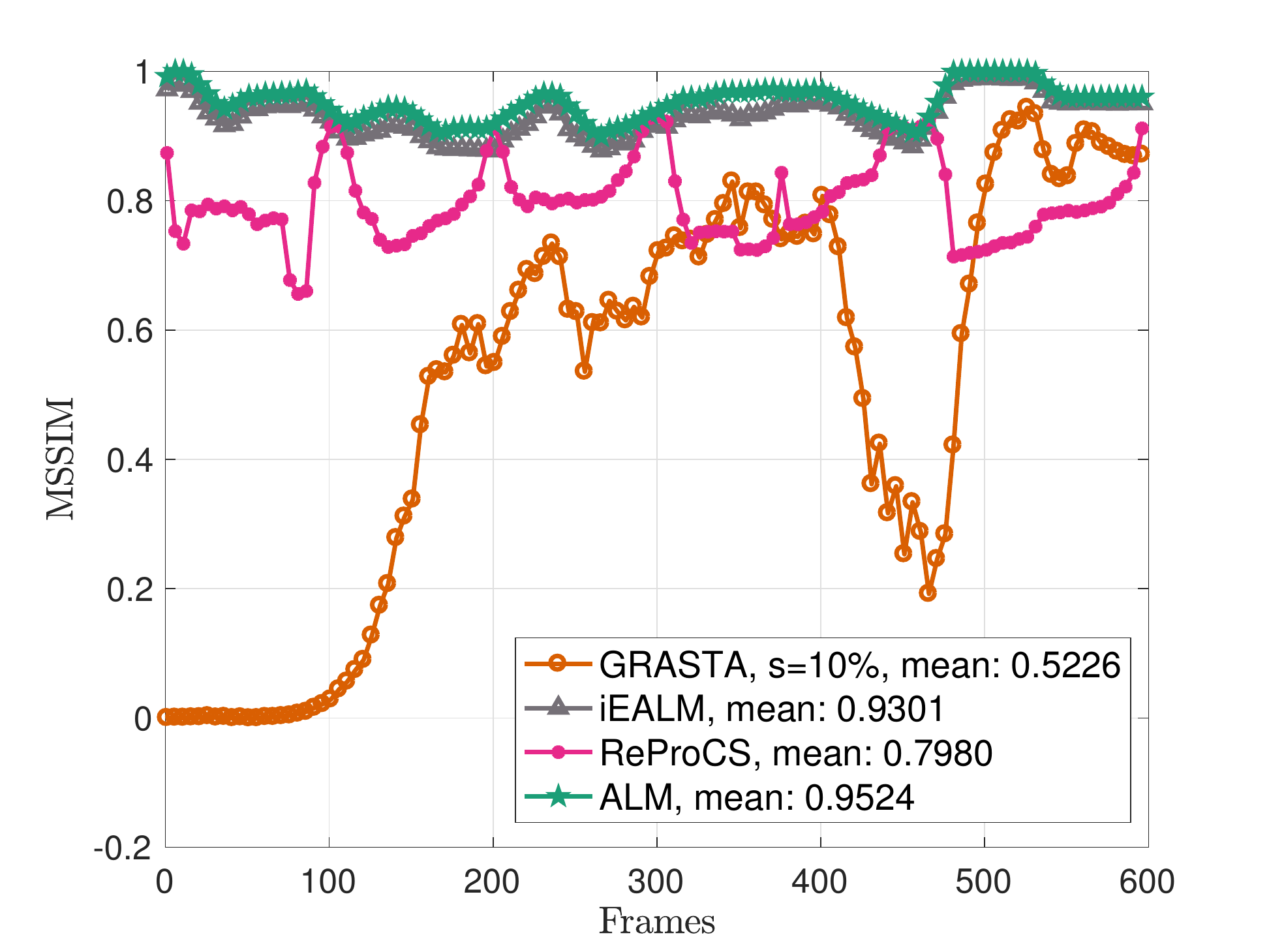}
    \caption{}
    \label{alm_ssim}
  \end{subfigure} 
  \caption{\small{(a) ROC curve to compare between ALM, iEALM, GRASTA, and ReProCS on {\tt Basic} video, frame size $144\times176$.} (b) \small{Comparison of Mean SSIM~(MSSIM) of ALM, iEALM, GRASTA, and ReProCS on~{\tt Basic} video. ALM has the best MSSIM.~To process 600 frames each of size $144\times176$, iEALM takes 164.03 seconds, GRASTA takes 20.25 seconds, ReProCS takes \textcolor{blue}{14.20} seconds, and  ALM takes~\textcolor{red}{13.13} seconds.}}
\end{figure*}
\subsection{Remarks on the Behaviour of ALM}

In this section, we propose the convergence of Algorithm~\ref{slm}. 
\begin{lemma}\label{bound:yk}
The sequence $\{{Y}^{(k)}\}$ is bounded.
\end{lemma}
\begin{proof}
By the optimality condition of $F_2^{(k+1)}$ we have,
$$
0\in\partial_{F_2}L(S^{(k+1)},F_2,Y^{(k)},\mu_k).
$$
Therefore, 
$$
0\in\partial\|F_2^{(k+1)}\|_{\ell_1}-\mu_k(A_2-QS^{(k+1)}-F_2^{(k+1)}+\frac{1}{\mu_k}Y^{(k)}), 
$$
which implies ${Y}^{(k+1)}\in\partial\|F_2^{(k+1)}\|_{\ell_1}.$~By using Theorem 4 in~\cite{LinChenMa}~(see also~\cite{watson}), we conclude that the sequence $\{Y^{(k)}\}$ is bounded by the dual norm of $\|\cdot\|_{\ell_1}$, that is, the $\|\cdot\|_{\infty}$ norm. 
\end{proof}

\begin{theorem}\label{Th:1}
There is a constant $\gamma$ such that
$$\|A_2-QS^{(k)}-F_2^{(k)}\|\le \frac{\gamma}{\mu_k},~~k=1,2,\cdots.$$
\end{theorem}
\begin{proof}
By using~\eqref{yk} we have 
\begin{eqnarray*}
A_2-QS^{(k)}-F_2^{(k)}=\frac{1}{\mu_{k-1}}(Y^{(k)}-Y^{(k-1)}).
\end{eqnarray*}
The result follows by applying Lemma~\ref{bound:yk}.

\end{proof}
\begin{figure}
 \centering
     \includegraphics[width = 0.4\textwidth]{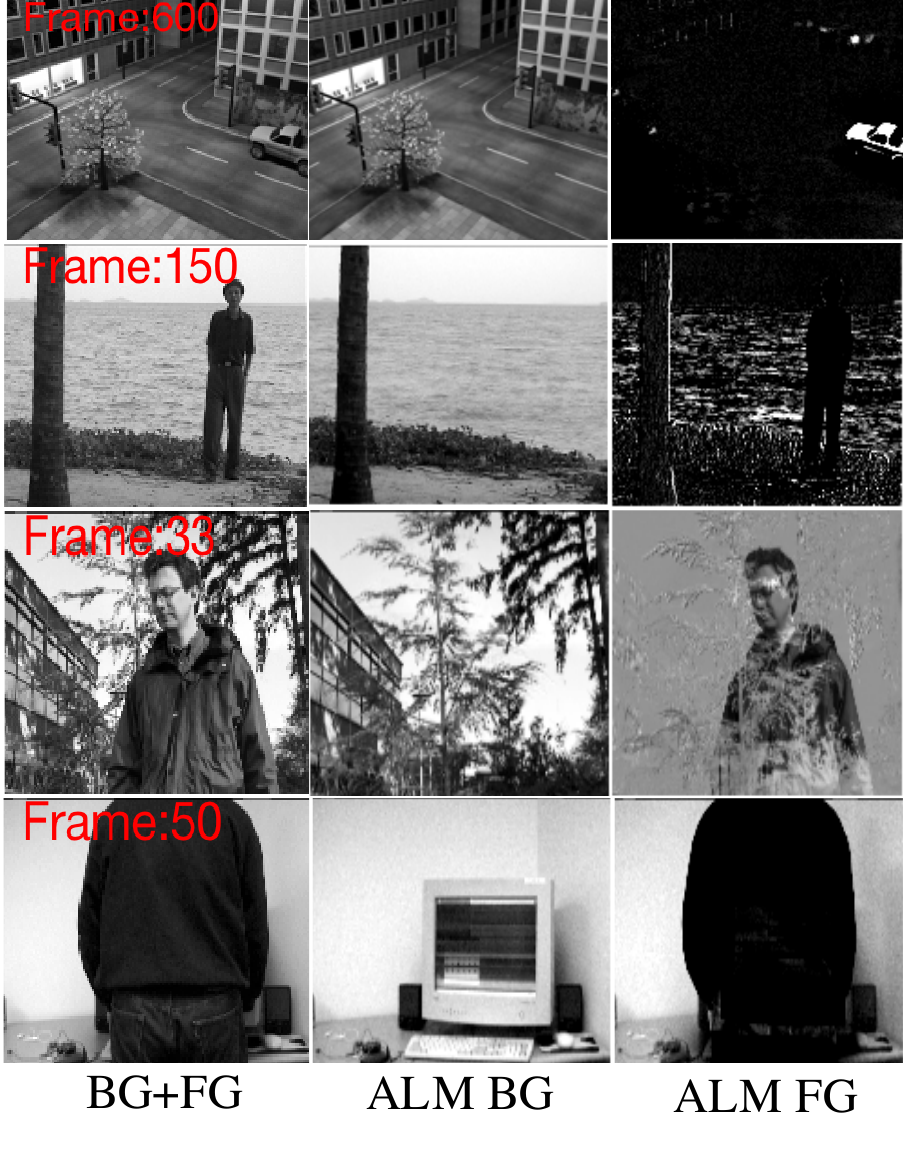}
    \caption{\small{Background and foreground recovered by ALM.~The videos have static foreground and dynamic background.}}
    \label{alm}
 \end{figure}

\begin{theorem} 
The sequence $\{L_{k}\}$ is bounded above and
$$
L_{k+1}-L_{k}\le O\left(\frac{1}{\mu_{k-1}}\right),~~k=1,2,\cdots.
$$
\end{theorem}

\begin{proof}We have, 
\begin{eqnarray}
L_{k+1} &=&L(S^{(k+1)},F_2^{(k+1)},Y^{(k)},\mu_{k})\nonumber\\
&\le& L(S^{(k+1)},F_2^{(k)},Y^{(k)},\mu_{k})\nonumber\\
&\le& L(S^{(k)},F_2^{(k)},Y^{(k)},\mu_{k})\nonumber\\
&=& \|F_2^{(k)}\|_{\ell_1}+\langle Y^{(k)}, A_2-QS^{(k)}-F_2^{(k)}\rangle\nonumber\\
&&+\frac{\mu_k}{2}\|A_2-QS^{(k)}-F_2^{(k)}\|_F^2\nonumber\\
&=& \|F_2^{(k)}\|_{\ell_1}+\langle Y^{(k-1)}, A_2-QS^{(k)}-F_2^{(k)}\rangle\nonumber\\
&&+\frac{\mu_{k-1}}{2}\|A_2-QS^{(k)}-F_2^{(k)}\|_F^2 \notag \\ 
&& \qquad +\langle Y^{(k)}-Y^{(k-1)}, A_2-QS^{(k)}-F_2^{(k)}\rangle\nonumber\\
&& +\frac{\mu_k-\mu_{k-1}}{2}\|A_2-QS^{(k)}-F_2^{(k)}\|_F^2\nonumber\\
&\overset{\text{(using~\eqref{yk})}}=&L_k+\mu_{k-1}\|A_2-QS^{(k)}-F_2^{(k)}\|_F^2\nonumber\\ 
&&+\frac{\mu_k-\mu_{k-1}}{2}\|A_2-QS^{(k)}-F_2^{(k)}\|_F^2\nonumber\\
&=&L_k+\frac{\mu_k+\mu_{k-1}}{2}\|A_2-QS^{(k)}-F_2^{(k)}\|_F^2.\nonumber
\end{eqnarray}
Therefore,
$$
L_{k+1}-L_{k}\le\frac{\mu_k+\mu_{k-1}}{2}\|A_2-QS^{(k)}-F_2^{(k)}\|_F^2,~~k=1,2,\cdots.
$$
By using \eqref{yk} we have for $k=1,2,\cdots$
\[
L_{k+1}-L_{k} \leq \frac{\mu_k+\mu_{k-1}}{\mu_{k-1}^2}\|Y^{(k)}-Y^{(k-1)}\|_F^2  =  \frac{1+\rho}{\mu_{k-1}}\|Y^{(k)}-Y^{(k-1)}\|_F^2.
\]
Next by using the boundedness of $\{{Y}^{(k)}\}$ we find
\begin{eqnarray*}
L_{k+1}-L_{k}&\le&O\left(\frac{1}{\mu_{k-1}}\right),~~k=1,2,\cdots,
\end{eqnarray*}
which is what we set out to prove.
\end{proof}

\begin{theorem} We have
$$
f^*-\|F_2^{(k)}\|_{\ell_1}\le O\left(\frac{1}{\mu_k}\right),
$$
where $f^*=\displaystyle{\min_{A_2=QS+F_2}\|F_2\|_{\ell_1}}.$
\end{theorem}
\begin{proof}
By using the triangle inequality we have
\begin{eqnarray}
\|F_2^{(k)}\|_{\ell_1}&\ge&\|A_2-QS^{(k)}\|_{\ell_1}\nonumber\\
&&-\|A_2-QS^{(k)}-F_2^{(k)}\|_{\ell_1}\nonumber\\
&\overset{\text{(using~\eqref{yk})}}\ge&f^*-\frac{1}{\mu_{k-1}}\|Y^{(k)}-Y^{(k-1)}\|_{\ell_1}.\nonumber\\
\end{eqnarray}
The result follows by applying boundedness of  the multipliers $Y^{(k)}$.
\end{proof}

\section{Smooth Optimization of $\ell_1$ Regression with Parallel Coordinate Descent Methods~\cite{Richtarik_Olivier}}

Imagine a situation when one processes a very low-resolution video sequence with a huge number of available training frames. That is, when there are more training frames $r$ than the number of pixels $m$, the method used in~\cite{Richtarik_Olivier} to solve~\eqref{dutta_richtarik_l1: decompose_problems} for each $i$ could be more effective. In this scenario we propose to solve each $\ell_1$ regression problem in~\eqref{dutta_richtarik_l1: decompose_problems} by using the parallel coordinate descent methods on their smooth variants~\cite{Richtarik_Olivier}. Note that each $f_i(s_i)$ is a non-smooth continuous convex function on a compact set $E_1$. By using Nesterov's smoothing technique~\cite{Nesterov} one can find a smooth approximation $f^{\mu}_i(s_i)$ of $f_i(s_i)$ for any $\mu>0.$~Fercoq et al.~\cite{Richtarik_Olivier} minimized $f^{\mu}_i(s_i)$ to approximately solve the original $\ell_1$ regression problem that contains $f_i(s_i)$.

\section{Additional numerical experiments demonstrating the effectiveness of ALM}
To demonstrate the robustness of the ALM in batch mode, we compare ALM with other state-of-the-art batch background estimation methods, such as, iEALM~\cite{LinChenMa} of RPCA, GRASTA~\cite{grasta}, and ReProCS~\cite{reprocs} on the {\tt Basic} scenario. We use 15 training frames for ALM. Figure~\ref{ROC_methods} shows that ALM covers the maximum area under the ROC curve.~Additionally, in Figure~\ref{alm_ssim}, our ALM has the best mean SSIM~(MSSIM) among all other methods.~Moreover, in batch mode, ALM takes the least computational time.~The background and foreground recovered by ALM in batch mode also shows its effectiveness in supervised background estimation~(see Figure~\ref{alm})

{\bibliographystyle{plain}
\bibliography{egbib}
}

\end{document}